\documentclass[10pt]{amsart}

\usepackage{amsthm,amsmath, amssymb}
\usepackage{pst-all}
\usepackage{xspace}
\usepackage{rotating}
\usepackage{epstopdf}

\theoremstyle{plain}
\newtheorem{prop}{Proposition}[section]
\newtheorem{coro}[prop]{Corollary}

\newtheorem{lemm}[prop]{Lemma}
\newtheorem{ques}[prop]{Question}

\theoremstyle{definition}
\newtheorem{defi}[prop]{Definition}
\newtheorem{exam}[prop]{Example}
\newtheorem{rema}[prop]{Remark}

\numberwithin{equation}{section}

\psset{unit=1mm}
\psset{dotsep=1.5pt}
\psset{dash=1.5pt 1.5pt}
\psset{linewidth=0.8pt}
\psset{arrowsize=3.5pt}
\psset{doublesep=1pt}
\psset{coilwidth=1.2mm}
\psset{coilheight=2}
\psset{coilaspect=20}
\psset{coilarm=2}
\newpsstyle{etc}{linestyle=dotted}
\newpsstyle{exist}{linestyle=dashed}
\newpsstyle{thin}{linewidth=0.5pt}
\newpsstyle{thinexist}{linewidth=0.5pt, linestyle=dashed}

\newcounter{ITEM}
\newcommand\ITEM[1]{\setcounter{ITEM}{#1}\leavevmode\hbox{\rm(\roman{ITEM})}}
\renewcommand\aa{a}
\newcommand\AAAA[2]{\Sigma_{#1,#2}^{\mathrm{0}}}
\newcommand\AF{T}
\newcommand\AFs{\AF^*}
\newcommand\AH{\Theta}
\newcommand\AHp{\AH'}
\newcommand\AHs{\AH^*}
\newcommand\bb{b}

\newcommand\BBBB[2]{\Sigma_{#1,#2}^{\mathrm{I}}}
\newcommand\Bi{\BR\infty}
\newcommand\BijA[2]{F_{#1,#2}^{\mathrm{0}}}
\newcommand\BijB[2]{F_{#1,#2}^{\mathrm{I}}}
\newcommand\BijC[2]{F_{#1,#2}^{\mathrm{II}}}
\newcommand\Bip{\BP\infty}
\newcommand\BP[1]{B_{#1}^+}
\newcommand\BVhat{\widehat{B\HS{-0.2}V}}
\newcommand\BR[1]{B_{#1}}
\newcommand\card[1]{\mathtt{\#}#1}
\newcommand\cc{c}
\newcommand\CCCC[2]{\Sigma_{#1,#2}^{\mathrm{II}_1}}
\newcommand\cl[1]{[#1]}

\newcommand\Cond{\diamondsuit}

\newcommand\dd{d}
\newcommand\DDDD[2]{\Sigma_{#1,#2}^{\mathrm{II}_2}}
\newcommand\dist{\mathsf{dist}}

\newcommand\dive{\mathrel{\preccurlyeq}\nobreak}
\newcommand\Dt[1]{\Delta_{#1}}
\newcommand\Dtt[1]{\underline\Delta_{#1}}

\newcommand\EF{\mathcal{E}_{F}}
\newcommand\EH{\mathcal{E}_{H}}
\newcommand\eqF{\eqp}
\newcommand\eqp{\equiv}
\newcommand\eqR{\equiv_\RRR}
\newcommand\ew{\varepsilon}
\newcommand\ff{f}
\newcommand\FF{F}
\newcommand\ffh{\widehat\ff}
\newcommand\FFp{F^+}
\newcommand\FFps{F^{+*}}
\newcommand\fft{\widetilde\ff}
\renewcommand\ge{\geqslant}
\newcommand\gf[1]{\tau_{#1}}
\renewcommand\gg{g}

\newcommand\ggh{\widehat\gg}
\newcommand\ggt{\widetilde\gg}
\newcommand\gh[1]{\theta_{#1}}

\newcommand\HH{H}
\newcommand\HHp{H^+}

\newcommand\HS[1]{\hspace{#1ex}}
\newcommand\HT[1]{\mathsf{ht}(#1)}

\newcommand\idf{\mathsf{id}}
\newcommand\ie{{\it i.e.}}
\newcommand\ii{i}

\newcommand\ind[1]{\mathsf{ind}(#1)}

\newcommand\inv{^{-1}}
\newcommand\jj{j}
\newcommand\kk{k}
\renewcommand\le{\leqslant}
\newcommand\lgg[1]{\vert#1\vert}
\newcommand\LI{\mathcal{O}}
\newcommand\LIp{\LI_{\Sigma}}
\newcommand\mm{m}
\newcommand\MM{M}
\newcommand\MON[2]{\langle#1\,\vert\,#2\rangle^+}
\newcommand\NF{\textsc{nf}}
\newcommand\nn{n}

\newcommand\NNNN{\mathbb{Z}_{\ge0}}
\newcommand\NPP[2]{N_{#1,#2}}

\newcommand\PF{\mathcal{P}_F}
\newcommand\PH{\mathcal{P}_H}
\newcommand\pdots{\HS{0.3}{\cdot}{\cdot}{\cdot}\HS{0.3}}
\newcommand\plf[1]{\lceil#1\rceil}
\newcommand\pp{p}
\newcommand\Pw{\mathfrak{P}}
\newcommand\qq{q}
\newcommand\redF[1]{\mathsf{red}(#1)}
\newcommand\redH[1]{\mathsf{red}(#1)}
\newcommand\resp{\mbox{\it resp., }}
\newcommand{\rev}{\curvearrowright}

\newcommand\RRR{\mathcal{R}}
\newcommand\sig[1]{\sigma_{\hspace{-0.2ex}#1}^{\null}}

\newcommand\SP[1]{\Sigma_{#1}}
\newcommand\SPP[2]{\Sigma_{#1,#2}}
\renewcommand\ss{s}

\newcommand\SSS{\mathcal{S}}
\newcommand\SSSs{\SSS^*}
\newcommand\Sym[1]{\mathfrak{S}_{#1}}
\newcommand\tI{\text{type}\ \mathrm{0}}
\newcommand\tII{\text{type}\ \mathrm{I}}
\newcommand\tIII{\text{type}\ \mathrm{II}_1}
\newcommand\tIV{\text{type}\ \mathrm{II}_2}
\newcommand\TO{\Rightarrow^{\HS{-0.3}*}}
\newcommand\toF{\Rightarrow}

\newcommand\toH{\Rightarrow_{\HS{-0.5}H}}
\newcommand\TOH{\Rightarrow_{\HS{-0.5}H}^*}

\newcommand\ttt{t}
\newcommand\uu{u}
\def\VR(#1,#2){\vrule width0pt height#1mm depth#2mm}
\newcommand\vv{v}
\newcommand\wdots{, ...\HS{0.2},}
\newcommand\ww{w}
\newcommand\WW{W}
\newcommand\xx{x}
\newcommand\XX{X}
\newcommand\yy{y}

\newcommand\ZZZZp{\mathbb{Z}_{>0}}

\author{Patrick DEHORNOY}

\address{Laboratoire de Math\'ematiques Nicolas Oresme UMR 6139\\ Universit\'e de Caen, 14032~Caen, France}
\email{patrick.dehornoy@unicaen.fr}
\urladdr{dehornoy.users.lmno.cnrs.fr}

\author{Emilie TESSON}
\address{Laboratoire de Math\'ematiques Nicolas Oresme UMR 6139\\ Universit\'e de Caen, 14032~Caen, France}

\title[GARSIDE COMBINATORICS FOR $\FFp$ AND A HYBRID WITH~$\Bip$]{GARSIDE COMBINATORICS FOR THOMPSON'S MONOID $\FFp$ AND A HYBRID WITH THE BRAID MONOID~$\Bip$}

\keywords{presented monoid; divisibility relation; simple elements; Thompson's group; braid group; normal form; Garside element; directed animal}

\subjclass{05E15, 20M05, 20E22, 68Q42}

\begin{document}

\begin{abstract}
On the model of simple braids, defined to be the left divisors of Garside's elements~$\Dt\nn$ in the monoid~$\Bip$, we investigate simple elements in Thompson's monoid~$\FFp$ and in a larger monoid~$\HHp$ that is a hybrid of~$\Bip$ and~$\FFp$: in both cases, we count how many simple elements left divide the right lcm of the first $\nn - 1$~atoms, and characterize their normal forms in terms of forbidden factors. In the case of~$\HHp$, a generalized Pascal triangle appears.
\end{abstract}

\maketitle

\section{Introduction}

Since the seminal work of F.A.\,Garside~\cite{Gar}, as extended in~\cite{Dlg} and~\cite{BrS}, it is known that Artin's braid group~$\BR\nn$ is a group of fractions for the monoid~$\BP\nn$ of positive $\nn$-strand braids and that the now called Garside element~$\Dt\nn$ plays a prominent role in the study of~$\BP\nn$. In particular, the divisors of~$\Dt\nn$ in~$\BP\nn$, called simple braids, form a family of~$\nn!$ elements in one-to-one correspondence with the permutations of $\{1 \wdots \nn\}$, leading to a remarkable combinatorics now at the heart of the algebraic study of~$\BR\nn$~\cite{Adj, Eps}, see~\cite[Chapter~IX]{Dir}. Subsequently, it was realised that such a situation can be found in many different contexts of groups and categories, always around a family of so-called simple elements resembling simple braids, and leading to various combinatorics, like, for instance, the dual Garside structure on~$\BR\nn$~\cite{BKL}, whose combinatorics is that of noncrossing partitions. 

Our aim in this paper is to investigate a Garside structure arising on Thompson's group~$\FF$~\cite{Tho, CFP} in connection with its submonoid~$\FFp$ generated by the standard (infinite) sequence of generators, corresponding to the presentation
\begin{equation}\label{E:PresF}
\FFp:= \MON{\gf1, \gf2, ...}{\gf\jj \gf\ii = \gf\ii \gf{\jj + 1} \quad\text{for}\quad \jj \ge \ii + 1}.
\end{equation}
To explain the similarity with braids and the natural questions in this non-finitely generated case, one should start from the infinite braid monoid
\begin{equation}\label{E:PresB}
\Bip = \bigg\langle \sig1, \sig2, ... \ \bigg\vert\ 
\begin{matrix}
\sig\jj \sig\ii = \sig\ii \sig\jj &\text{for} &\jj \ge \ii + 2\\
\sig\jj \sig\ii \sig\jj = \sig\ii \sig\jj \sig\ii &\text{for} &\jj = \ii + 1
\end{matrix}
\ \bigg\rangle^+:
\end{equation}
in this case, Garside's braid~$\Dt\nn$ is the right lcm of the $\nn - 1$~first atoms~$\sig1 \wdots \sig{\nn - 1}$ of~$\Bip$ (see Section~\ref{SS:Term} for a reminder about the terminology), and simple braids are those braids that left divide at least one element~$\Dt\nn$ in~$\Bip$. 

In the case of the monoid~$\FFp$, the atoms are the elements~$\gf\ii$, and we shall see that there exists for each~$\nn$ a well defined element~$\Dt\nn$ that is, in~$\FFp$, the right lcm of the first $\nn - 1$~atoms. Then we shall investigate the derived simple elements, namely the elements of~$\FFp$ that left divide at least one element~$\Dt\nn$. The main results proved here are that, for every~$\nn$, there exist $2^{\nn - 1}$~simple elements left dividing~$\Dt\nn$ in~$\FFp$, in explicit one-to-one correspondence with the subsets of~$\{1 \wdots \nn - 1\}$, and that simple elements form a Garside family in~$\FFp$\cite[Def.~III.1.31]{Dir}, thus guaranteeing the existence and properties of an associated greedy normal form in~$\FFp$. These results are established by combining the existence of a convergent rewrite system on~$\FFp$ and the reversing technique~\cite{Dff, Dia} for analyzing the divisibility relations of a presented monoid.

The above results are technically easy, and we then switch to a combinatorially more involved situation related to another monoid~$\HHp$, which is a hybrid of the braid monoid~$\Bip$ and the Thompson monoid~$\FFp$. Various hybrids of the groups~$\Bi$ and~$\FF$ have already been considered, in particular the group~$\BVhat$ of~\cite{Bri1, Bri2, Dhe}, which is a group of fractions for a monoid, that is a Zappa-Sz\'ep product of~$\FFp$ and~$\Bip$ and, therefore, inherits their Garside structures. Here we shall introduce and investigate a new hydrid, which is not a product but rather a mixture of the initial monoids~$\FFp$ and~$\Bip$. Indeed, we consider
\begin{equation}\label{E:PresH}
\HHp := \bigg\langle \gh1,\gh2, ... \ \bigg\vert\ 
\begin{matrix}
\gh\jj \gh\ii = \gh\ii \gh{\jj + 1} &\text{for} &\jj \ge \ii + 2\\
\gh\jj \gh\ii \gh\jj = \gh\ii \gh\jj \gh{\ii + 3} &\text{for} &\jj = \ii + 1
\end{matrix}
\ \bigg\rangle^{\!\scriptstyle+},
\end{equation}
in which the length~$2$ relations are Thompson's relations as in~\eqref{E:PresF}, whereas the length~$3$ relations are directly reminiscent of braid relations of~\eqref{E:PresB}, but with a shift of one index. Here, we investigate the basic properties of the monoid~$\HHp$ and, specifically, the associated Garside combinatorics, if this makes sense. Actually, it does: we shall see that, for every~$\nn$, the atoms~$\gh1 \wdots \gh{\nn - 1}$ admit a right lcm, again denoted by~$\Dt\nn$, so that it is natural to investigate simple elements, defined to be those that left divide some element~$\Dt\nn$. The main results proved here are that, for every~$\nn$, there exist $2 \cdot 3^{\nn - 2}$~simple elements left dividing~$\Dt\nn$ in~$\HHp$, with an explicit description of a distinguished expression for each of them. As in the case of~$\FFp$, these results are established using a convergent rewrite system on~$\HHp$ and the reversing technique; the proofs are more difficult than for~$\FFp$ and some of them require delicate inductive arguments. We hope that the existence of this nontrivial combinatorics will draw some attention to the monoid~$\HHp$, and to the group~$\HH$ presented by~\eqref{E:PresH}, which remains essentially mysterious.

The paper is divided into four sections after this introduction. In Section~\ref{S:Fp}, we investigate the monoid~$\FFp$ and the derived simple elements, providing a good warm-up for the sequel. In Section~\ref{S:Hp}, we establish various general properties of the monoid~$\HHp$, in particular the fact that it admits cancellation on both sides. Next, in Section~\ref{S:Delta}, we study the elements~$\Dt\nn$ of~$\HHp$ and count their left divisors by partitioning them into several families. Finally, in Section~\ref{S:NF}, we explicitly characterize the normal form (in the sense of some convergent rewrite system) of simple elements of~$\HHp$.

\subsection*{Acknowledgement}
The authors thank Matthieu Picantin for having pointed at the connection between the numbers~$\NPP\kk\ell$ of Section~\ref{SS:Part} and directed animals.

\section{Thompson's monoid~$\FFp$}\label{S:Fp}

Here we study the case of Thompson's monoid~$\FFp$, an easy first step. It is standard that \eqref{E:PresF} is a presentation of Thompson's group~$\FF$, and, as the relations involve no inverse of the generators, it makes sense to introduce the associated monoid~$\FFp$ and to consider the associated Garside combinatorics, if it exists. 

The section is divided into four parts. In Section~\ref{SS:Term}, we recall the standard terminology for the divisibility relations in a monoid, extensively used throughout the text. Next, in Section~\ref{SS:NFF}, we define a convergent rewrite system that selects a distinguished expression for every element of~$\FFp$. In Section~\ref{SS:RevF}, we recall basic notions about word reversing, here in the new version of~\cite{Djb}, and use them to show that $\FFp$ is cancellative and admits right lcms (least common right multiples). Finally, in Section~\ref{SS:GarF}, we investigate the elements~$\Dt\nn$ and describe their left divisors explicitly.

\subsection{The divisibility relations of a monoid}\label{SS:Term}

Let $\MM$ be a monoid (possibly, in particular, a free one, \ie, a monoid of words). For~$\aa, \bb$ in~$\MM$, we say that $\aa$ \emph{left divides}~$\bb$ in~$\MM$, or, equivalently, that $\bb$ is a \emph{right multiple} of~$\aa$, written $\aa \dive \bb$, if $\aa\xx = \bb$ holds for some~$\xx$ (of~$\MM$). If $\MM$ is left cancellative (meaning that $\xx\aa = \xx\bb$ implies $\aa = \bb$) and $1$ is the only invertible element in~$\MM$, the relation~$\dive$ is a partial ordering on~$\MM$. 

For~$\aa, \bb$ in~$\MM$, we say that $\cc$ is a \emph{right lcm} (least common right multiple) of~$\aa$ and~$\bb$ if $\aa \dive \cc$ and $\bb \dive \cc$ hold, and the conjunction of $\aa \dive \xx$ and~$\bb \dive \xx$ implies~$\cc \dive \xx$: in other words, $\cc$ is a lowest upper bound of~$\aa$ and~$\bb$ with respect to~$\dive$. 

The symmetric notions of a right divisor and a left multiple are defined similarly, replacing $\aa\xx = \bb$ with $\xx\aa = \bb$. Finally, we say that $\aa$ is a \emph{factor} of~$\bb$ if $\xx\aa\yy = \bb$ holds for some~$\xx, \yy$.

An element~$\aa$ of~$\MM$ is said to be an \emph{atom} if it admits no decomposition $\aa = \bb\cc$ with $\bb \not= 1$ and~$\cc \not= 1$.

\subsection{A normal form on~$\FFp$}\label{SS:NFF}

We begin our investigation of the monoid~$\FFp$. We recall that $\FFp$ is defined by the presentation
\begin{equation*}
\FFp:= \MON{\gf1, \gf2, ...}{\gf\jj \gf\ii = \gf\ii \gf{\jj + 1} \quad\text{for}\quad \jj \ge \ii + 1},
\end{equation*}
hereafter denoted by~$\PF$. We put $\AF:= \{\gf\ii \mid \ii \ge 1\}$, write $\AFs$ for the free monoid of all words in the alphabet~$\AF$, and~$\eqF$ for the congruence on~$\AFs$ generated by the relations of~$\PF$. We use~$\ew$ for the empty word. Our first tool for studying~$\FFp$ consists in defining a unique normal form using a rewrite system on~$\AFs$.

\begin{lemm}\label{L:ConvF}
Let~$\EF$ be the rewrite system on~$\AF^*$ defined by the rules
\begin{equation}\label{E:RewF}
\gf\ii \gf{\jj + 1} \to \gf{\jj } \gf\ii \text{\quad for $\ii \ge 1$ and $\jj \ge \ii + 1$}.
\end{equation}
Then $\EF$ is convergent.
\end{lemm}

\begin{proof}
As is standard, see for instance~\cite{Ter}, we shall check that~$\EF$ is noetherian and locally confluent. We write~$\toF$ for the one-step rewrite relation associated with the rules of~\eqref{E:RewF}, that is, for the family of all pairs $$(\ww_1 \gf\ii \gf{\jj + 1} \ww_2\ ,\ \ww_1 \gf{\jj} \gf\ii \ww_2) \quad\text{with $\jj \ge \ii + 1$},$$
and~$\TO$ for the reflexive--transitive closure of~$\toF$. For~$\ww$ in~$\AF^*$, let $\rho(\ww)$ be the sum of the indices of the generators~$\gf\ii$ occurring in~$\ww$. Then $\ww \toF \ww'$ implies $\rho(\ww) > \rho(\ww')$, and, therefore, there is no proper infinite sequence for~$\toF$. So $\EF$ is noetherian.

Next, assume $\ww \toF \ww'$ and $\ww \toF \ww''$. By definition, $\ww'$ and $\ww''$ are obtained from~$\ww$ by replacing some length~$2$ factor~$\gf\ii \gf{\jj + 1}$ with the corresponding word~$\gf\jj \gf\ii$. For local confluence, the case of disjoint factors is trivial, and the critical case of overlapping factors corresponds to $\ww = \gf\ii \gf{\jj + 1} \gf{\kk + 2}$ with $\jj \ge \ii + 1$ and~$\kk \ge \jj + 1$, leading to $\ww' = \gf{\jj} \gf\ii \gf{\kk + 2}$ and $\ww'' = \gf\ii \gf{\kk + 1} \gf{\jj + 1}$. One then obtains
\begin{equation}\label{E:LocConfF}
\VR(7,7)\begin{picture}(55,0)(0,6)
\put(-5,6){$\gf\ii \gf{\jj + 1} \gf{\kk + 2}$}
\put(21,12){$\gf{\jj} \gf\ii \gf{\kk + 2}$}
\put(19,0){$\gf\ii \gf{\kk + 1} \gf{\jj + 1}$}
\put(45,6){$\gf{\kk} \gf{\jj} \gf\ii$\ ,}
\put(13,3){\rotatebox{-30}{\hbox{$\toF$}}}
\put(13,8.5){\rotatebox{30}{\hbox{$\toF$}}}
\put(36,1.5){\rotatebox{30}{\hbox{$\toF^2$}}}
\put(36,10.5){\rotatebox{-30}{\hbox{$\toF^2$}}}
\end{picture}
\end{equation}
It follows that $\EF$ is locally confluent, hence convergent by Newman's diamond lemma~\cite{New}.
\end{proof}

For every word~$\ww$ of~$\AFs$, we shall denote by~$\redF\ww$ the unique $\EF$-reduced word~$\ww'$ satisfying $\ww \TO \ww'$. By definition, the words~$\ww$ and~$\redF\ww$ represent the same element of~$\FFp$, and $\redF\ww$ is the unique $\EF$-reduced word in the equivalence class of~$\ww$ in~$\FFp$. Thus, Lemma~\ref{L:ConvF} implies

\begin{prop}\label{P:RedF}
$\EF$-reduced words provide a unique normal form for the elements of the monoid~$\FFp$.
\end{prop}

It directly follows from the definition that a word of~$\AFs$ is $\EF$-reduced if, and only if, it has no length~$2$ factor~$\gf\ii \gf{\jj + 1}$ with $\jj \ge \ii + 1$, which implies that, for every~$\nn$, the set of $\EF$-reduced words lying in~$\{\gf1 \wdots \gf\nn\}^*$ is a regular language~\cite{Eps, HoR}.

\subsection{Using word reversing}\label{SS:RevF}

The second method for investigating the monoid~$\FFp$ is word reversing~\cite{Dia}, a distillation of an argument that ultimately stems from Garside's approach to braid monoids~\cite{Gar}. Here we shall describe reversing using the new formalism of~\cite{Djb}, which is specially convenient in the current case (and in that of~$\HHp$ in Section~\ref{SS:RevH}). So we introduce reversing as a binary relation on pairs of words connected with a particular type of van Kampen diagram.

\begin{defi}\cite{Djb}\label{D:Rev}
A \emph{reversing grid} for a monoid presentation~$(\SSS, \RRR)$, or \emph{$(\SSS, \RRR)$-grid}, is a rectangular diagram consisting of finitely many matching $\SSS \cup \{\ew\}$-labeled pieces of the types

- \VR(12,9)\begin{picture}(34,0)(-3,6)
\pcline{->}(1,13)(21,13)\taput{$\ttt$}
\pcline{->}(0,12)(0,1)\tlput{$\ss$}
\pcline{->}(1,0)(7,0)\taput{$\ttt_1$}
\pcline[style=etc](8,0)(14,0)
\pcline{->}(15,0)(21,0)\taput{$\ttt_\qq$}
\pcline{->}(22,12)(22,8)\trput{$\ss_1$}
\pcline[style=etc](22,7)(22,6)
\pcline{->}(22,5)(22,1)\trput{$\ss_\pp$}
\end{picture} 
\parbox{90mm}{with $\ss, \ttt, \ss_1 \wdots \ss_\pp, \ttt_1 \wdots \ttt_\qq$ in~$\SSS$\\ \null\hspace{20mm}and $\ss\ttt_1 {\pdots} \ttt_\qq = \ttt \ss_1 {\pdots} \ss_\pp$ a relation of~$\RRR$,} 

- \VR(6,8)\begin{picture}(16,0)(-3,4)
\pcline{->}(1,8)(9,8)\taput{$\ss$}
\pcline{->}(0,7)(0,1)\tlput{$\ss$}
\pcline{->}(1,0)(9,0)\tbput{$\ew$}
\pcline{->}(10,7)(10,1)\trput{$\ew$}
\end{picture}, \quad
\begin{picture}(16,0)(-3,4)
\pcline{->}(1,8)(9,8)\taput{$\ew$}
\pcline{->}(0,7)(0,1)\tlput{$\ss$}
\pcline{->}(1,0)(9,0)\tbput{$\ew$}
\pcline{->}(10,7)(10,1)\trput{$\ss$}
\end{picture}, \quad
\begin{picture}(16,0)(-3,4)
\pcline{->}(1,8)(9,8)\taput{$\ttt$}
\pcline{->}(0,7)(0,1)\tlput{$\ew$}
\pcline{->}(1,0)(9,0)\tbput{$\ttt$}
\pcline{->}(10,7)(10,1)\trput{$\ew$}
\end{picture}, \quad
\begin{picture}(16,0)(-3,4)
\pcline{->}(1,8)(9,8)\taput{$\ew$}
\pcline{->}(0,7)(0,1)\tlput{$\ew$}
\pcline{->}(1,0)(9,0)\tbput{$\ew$}
\pcline{->}(10,7)(10,1)\trput{$\ew$}
\end{picture}\quad
with $\ss, \ttt$ in~$\SSS$.

\noindent For $\uu, \vv, \uu_1, \vv_1$ in~$\SSSs$, we say that an $(\SSS, \RRR)$-grid~$\Gamma$ goes \emph{from~$(\uu, \vv)$ to~$(\uu_1, \vv_1)$} if the labels of the left and top edges of~$\Gamma$ form the words~$\uu$ and~$\vv$, respectively, whereas the labels of the right and bottom edges form the words~$\uu_1$ and~$\vv_1$. We write $(\uu, \vv) \rev_\RRR (\uu_1, \vv_1)$ if there exists a $(\SSS, \RRR)$-grid from~$(\uu, \vv)$ to~$(\uu_1, \vv_1)$.
\end{defi}

\begin{exam}\label{X:Rev}
Two typical $\PF$-grids are
\begin{equation}\label{E:Grid}
\VR(9,6)\begin{picture}(25,0)(0,4)
\pcline{->}(1,10)(11,10)\taput{$\gf1$}
\pcline{->}(13,10)(23,10)\taput{$\gf3$}
\pcline{->}(1,0)(11,0)\tbput{$\gf1$}
\pcline{->}(13,0)(23,0)\tbput{$\ew$}
\pcline{->}(0,9)(0,1)\tlput{$\gf2$}
\pcline{->}(12,9)(12,1)\trput{$\gf3$}
\pcline{->}(24,9)(24,1)\trput{$\ew$,}
\end{picture}
\hspace{20mm}
\begin{picture}(25,0)(0,4)
\pcline{->}(1,10)(11,10)\taput{$\gf2$}
\pcline{->}(13,10)(23,10)\taput{$\gf1$}
\pcline{->}(1,0)(11,0)\tbput{$\ew$}
\pcline{->}(13,0)(23,0)\tbput{$\gf1$}
\pcline{->}(0,9)(0,1)\tlput{$\gf2$}
\pcline{->}(12,9)(12,1)\trput{$\ew$}
\pcline{->}(24,9)(24,1)\trput{$\ew$,}
\end{picture}
\end{equation}
witnessing for the relations $(\gf2, \gf1\gf3) \rev (\ew, \gf1)$ and $(\gf2, \gf2\gf1) \rev (\ew, \gf1)$, respectively---we omit the index in~$\rev$ when there is no ambiguity. Note that, because all relations of~$\PF$ involve words of length~$2$, the pieces of the first type in Definition~\ref{D:Rev} are squares: the right and bottom edges each consist of one single $\SSS$-labeled arrow.
\end{exam}

The following result is (a special case of a result) established in~\cite{Djb}. Below we say that a monoid presentation~$(\SSS, \RRR)$ is \emph{homogeneous} if every relation in~$\RRR$ has the form $\ww = \ww'$ with $\ww, \ww'$ of the same length, and \emph{right complemented} if it contains no relation $\ss... = \ss...$ and at most one relation $\ss... = \ttt...$ for all $\ss \not= \ttt$ in~$\SSS$. On the other hand, two $(\SSS, \RRR)$-grids~$\Gamma$ from~$(\uu, \vv)$ to~$(\uu_1, \vv_1)$ and $\Gamma'$ from~$(\uu', \vv')$ to~$(\uu'_1, \vv'_1)$ are \emph{equivalent} if we have $\uu' \eqR \uu$, $\vv' \eqR \vv$, $\uu'_1 \eqR \uu_1$, and~$\vv'_1 \eqR \vv_1$, where $\eqR$ is the congruence on~$\SSSs$ generated by~$\RRR$---so that the monoid~$\MON\SSS\RRR$ is~$\SSSs{/}{\eqR}$.

\begin{lemm}\cite[Propositions~1.12, 1.14, 1.16]{Djb}\label{L:Rev}
Assume that $(\SSS, \RRR)$ is a homogeneous right complemented monoid presentation and, for every~$\ss$ in~$\SSS$ and every relation $\ww = \ww'$ in~$\RRR$,
\begin{equation}\label{E:CompatS}\tag{$\Cond$}
\parbox{113mm}{for every grid from~$(\ss, \ww)$, there is an equivalent grid from~$(\ss, \ww')$, \\ \null\hfill and vice versa.}
\end{equation}

\ITEM1 Two words~$\uu, \vv$ of~$\SSSs$ represent the same element of the monoid~$\MON\SSS\RRR$ if, and only if, $(\uu, \vv) \rev (\ew, \ew)$ holds.

\ITEM2 The monoid~$\MON\SSS\RRR$ is left cancellative.

\ITEM3 Two elements~$\aa, \bb$ of~$\MON\SSS\RRR$ represented by~$\uu$ and~$\vv$ in~$\SSSs$ admit a common right multiple if, and only if, $(\uu, \vv) \rev_\RRR (\uu_1, \vv_1)$ holds for some~$\uu_1, \vv_1$; in this case, the element represented by~$\uu\vv_1$ is a right lcm of~$\aa$ and~$\bb$. In the special case when, for all~$\ss \not= \ttt$ in~$\SSS$, there exist~$\ss', \ttt'$ in~$\SSS$ such that $\ss\ttt' = \ttt\ss'$ is a relation of~$\RRR$, there always exist~$\uu_1, \vv_1$ as above, and any two elements of~$\MON\SSS\RRR$ admit a right lcm.
\end{lemm}

Applying Lemma~\ref{L:Rev}, we deduce:

\begin{prop}\label{P:LcmF}
The monoid~$\FFp$ is left and right cancellative. Any two elements of~$\FFp$ admit a right lcm. Any two elements of~$\FFp$ that admit a common left multiple admit a left lcm.
\end{prop}

\begin{proof}
In view of applying Lemma~\ref{L:Rev}, we observe that the presentation~$\PF$ is homogeneous (all relations are of the form $\ww = \ww'$ with $\ww$ and~$\ww'$ of length two), right complemented with one relation $\gf\ii... = \gf\jj...$ for all~$\ii, \jj$, and that Condition~\eqref{E:CompatS} holds for every~$\gf\ii$ and every relation of~$\PF$. To this end, we consider all pairs $(\gf\ii, \gf\jj \gf{\kk + 1})$ with $\kk \ge \jj + 1$, and compare the reversing grids from~$(\gf\ii, \gf\jj \gf{\kk + 1})$ and from~$(\gf\ii, \gf\kk \gf\jj)$: the two grids of Example~\ref{X:Rev} are typical, corresponding to $\ii = 2$, $\jj =\nobreak 1$, and~$\kk = 2$, and they are indeed equivalent, since both admit as output~$(\ew, \gf1)$. The number of triples~$(\ii, \jj, \kk)$ to consider is infinite but only finitely many patterns may occur, according to the position of~$\ii$ with respect to~$\jj$ and~$\kk$. We skip the details, which are fairly obvious. Having established~\eqref{E:CompatS}, we deduce from Lemma~\ref{L:Rev}\ITEM2 that the monoid~$\FFp$ is left cancellative and from Lemma~\ref{L:Rev}\ITEM3 that any two elements of~$\FFp$ admit a right lcm.

To study left multiples, we observe that the presentation~$\PF$ is also left complemented (in the obvious sense), and consider the notion of a left reversing grid, which is symmetric to the above notion of a right reversing grid (which amounts to considering the opposed monoid). To this end, we replace each elementary diagram \VR(12,9)
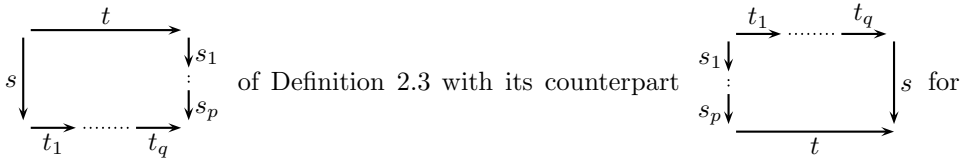
\begin{picture}(29,0)(-2,5)
\pcline{->}(1,13)(21,13)\taput{$\ttt$}
\pcline{->}(0,12)(0,1)\tlput{$\ss$}
\pcline{->}(1,0)(7,0)\tbput{$\ttt_1$}
\pcline[style=etc](8,0)(14,0)
\pcline{->}(15,0)(21,0)\tbput{$\ttt_\qq$}
\pcline{->}(22,12)(22,8)\trput{$\ss_1$}
\pcline[style=etc](22,7)(22,6)
\pcline{->}(22,5)(22,1)\trput{$\ss_\pp$}
\end{picture} 
of Definition~\ref{D:Rev} with its counterpart \VR(12,9)\begin{picture}(30,0)(-5,5.5)
\pcline{->}(1,0)(22,0)\tbput{$\ttt$}
\pcline{->}(22,12)(22,1)\trput{$\ss$}
\pcline{->}(1,13)(7,13)\taput{$\ttt_1$}
\pcline[style=etc](8,13)(14,13)
\pcline{->}(15,13)(21,13)\taput{$\ttt_\qq$}
\pcline{->}(0,12)(0,8)\tlput{$\ss_1$}
\pcline[style=etc](0,7)(0,6)
\pcline{->}(0,5)(0,1)\tlput{$\ss_\pp$}
\end{picture}
for $\ss_1 {\pdots} \ss_\pp \ttt = \ttt_1 {\pdots} \ttt_\qq \ss$ in~$\RRR$ and, similarly, replace \VR(6,6)\begin{picture}(16,0)(-3,3)
\pcline{->}(1,8)(9,8)\taput{$\ss$}
\pcline{->}(0,7)(0,1)\tlput{$\ss$}
\pcline{->}(1,0)(9,0)\tbput{$\ew$}
\pcline{->}(10,7)(10,1)\trput{$\ew$}
\end{picture}
with \VR(6,6)\begin{picture}(16,0)(-3,3)
\pcline{->}(1,8)(9,8)\taput{$\ew$}
\pcline{->}(0,7)(0,1)\tlput{$\ew$}
\pcline{->}(1,0)(9,0)\tbput{$\ss$}
\pcline{->}(10,7)(10,1)\trput{$\ss$.}
\end{picture}
Then one easily checks that the counterpart of~\eqref{E:CompatS} is satisfied and one deduces, by the counterpart of Lemma~\ref{L:Rev}\ITEM2, that~$\FFp$ is right cancellative. Finally, the counterpart of Lemma~\ref{L:Rev}\ITEM3 implies that any two elements of~$\FFp$ that admit a common left multiple admit a left lcm. However, two elements of~$\FFp$ need not always admit a common left multiple: there is no relation $...\gf1 = ...\gf2$ in~$\PF$, and, therefore, the counterpart of Lemma~\ref{L:Rev}\ITEM3 implies that $\gf1$ and $\gf2$ admit no common left-multiple in~$\FFp$.
\end{proof}

It follows from Proposition~\ref{P:LcmF} and Ore's classical theorem~\cite{Ore} that the monoid~$\FFp$ embeds in its enveloping group, which is the group presented by~$\PF$, namely Thompson's group~$\FF$, and that the latter is a group of right fractions for~$\FFp$, that is, every element of~$\FF$ can be expressed as~$\aa \bb\inv$ with~$\aa, \bb$ in~$\FFp$. The expression is unique if, in addition, we require that the fraction be irreducible, meaning that $\aa$ and~$\bb$ admit no common right divisor.

\begin{rema}
As explained in~\cite{Dic}, there exists a (more redundant) positive presentation~$\PF^*$ of the group~$\FF$ in terms of a family of generators~$\gf\ss^*$ with~$\ss$ a finite sequence of~$0$s and~$1$s such that $\gf\ii$ coincides with~$\gf{1^{\ii - 1}}^*$ and that $\FF$ is a group both of left and right fractions for the monoid~$\FFps$ defined by~Ê$\PF^*$. The latter admits left and right lcms and is a sort of counterpart for the dual braid monoid of~\cite{BKL}. The main relations in~$\PF^*$ correspond to the MacLane--Stasheff pentagon.
\end{rema}

\subsection{Garside combinatorics for~$\FFp$}\label{SS:GarF}

The monoid~$\FFp$ resembles the braid mon\-oid~$\Bip$ in that it is cancellative and admits right lcms and, therefore, it makes sense to consider the counterpart of the Garside elements~$\Dt\nn$ and their divisors.

As the presentation~$\PF$ is homogeneous, the atoms of~$\FFp$ are the elements~$\gf\ii$ with $\ii \ge 1$. So, exactly as in the case of~$\Bip$, we shall consider the element~$\Dt\nn$ that is the right lcm of~$\gf1 \wdots \gf{\nn - 1}$---we might use a different notation, for instance~$\Dt\nn^\FF$, but there will be no risk of ambiguity here. We start from an explicit expression.

\begin{defi}\label{D:DeltaF}	
We put $\Dtt1 := \ew$, and, for $\nn \ge 2$, we put $\Dtt\nn := \gf1 \gf3 \gf5 \pdots \gf{2\nn - 3}$. We denote by~$\Dt\nn$ the class of~$\Dtt\nn$ in~$\FFp$.
\end{defi}

It is clear that $\Dt\nn$ left divides~$\Dt{\nn + 1}$ for each~$\nn$, and one inductively checks that the $\EF$-normal form of~$\Dt\nn$ is $\gf{\nn - 1} \gf{\nn - 2} \pdots \gf2 \gf1$.

\begin{lemm}\label{L:AtomDivF}
For every~$\nn \ge 2$, the element~$\Dt\nn$ is the right lcm of~$\gf1 \wdots \gf{\nn-1}$. No element~$\gf\ii$ with~$\ii \ge \nn$ left divides~$\Dt\nn$.
\end{lemm}

\begin{proof}
We prove using induction on~$\nn \ge 2$ that $\Dt\nn$ is the right lcm of~$\gf1 \wdots \gf{\nn - 1}$. The result is trivial for $\nn = 2$. Assume $\nn \ge 3$. A direct computation gives
$$(\gf{\nn - 1}, \Dtt{\nn - 1}) \rev (\gf{2\nn - 3}, \Dtt{\nn - 1}).$$
By Lemma~\ref{L:Rev}\ITEM3, this implies that $\Dtt\nn$ represents the right lcm of~$\gf{\nn - 1}$ and~$\Dt{\nn - 1}$. By induction hypothesis, $\Dt{\nn - 1}$ is the right lcm of~$\gf1 \wdots \gf{\nn - 2}$, so $\Dt\nn$ is the right lcm of~$\gf1 \wdots \gf{\nn - 1}$. On the other hand, for $\ii \ge \nn$, we find $(\gf\ii, \Dtt\nn) \rev (\gf{\ii + \nn - 1}, \Dtt\nn)$, which shows that the right lcm of~$\gf\ii$ and~$\Dt\nn$ is not~$\Dt\nn$, so $\gf\ii$ does not left divide~$\Dt\nn$. 
\end{proof}

The main notion in Garside theory~\cite{Dir} is the notion of a simple element, defined as the (left) divisors of the distinguished element(s)~$\Delta$. 

\begin{defi}\label{D:SimpleF}
An element~$\aa$ of~$\FFp$ is called \emph{simple} if $\aa \dive \Dt\nn$ holds for some~$\nn$. 
\end{defi}

Our aim is to understand the structure of simple elements of~$\FFp$, typically to characterize their normal forms. To this end, the key point will be the following exhaustive description of the expressions of~$\Dt\nn$. Below, we write~$\Sym\nn$ for the group of all permutations of~$\{1 \wdots \nn\}$, and $\ss_\ii$ for the transposition~$(\ii, \ii+1)$.

\begin{lemm}\label{L:ExpDtF}
The expressions of~$\Dt\nn$ are the words~$\ww_\ff$ with~$\ff$ in~$\Sym{\nn - 1}$, where, for~$\pp \le \nn - 1$, we put 
$$\ffh(\pp) := \card\{\ii < \ff\inv(\pp) \mid \ff(\ii) > \pp\} \quad \text{and} \quad \fft(\pp) := 2\ff\inv(\pp) - 1 - \ffh(\pp),$$
and let $\ww_\ff$ be the word $\gf{\fft(1)} \gf{\fft(2)} \pdots \gf{\fft(\nn - 1)}$.
\end{lemm}

\begin{proof}
We first establish the following technical result:
\begin{equation}\label{E:ExpDtF}
\parbox{110mm}{%
If $\ff\inv(\pp) < \ff\inv(\pp + 1)$ (\resp $>$) holds, then so does $\fft(\pp) + 1 < \fft(\pp + 1)$ (\resp $\fft(\pp) > \fft(\pp + 1)$); applying a relation of~$\PF$ to~$\ww_\ff$ in position~$\pp$ yields \VR(3.4,0) the word~$\ww_{\ss_\pp\ff}$.}
\end{equation}
So assume $\ff\inv(\pp +1) = \ff\inv(\pp) + \mm$ with $\mm \ge 1$. The definition gives
$$\ffh(\pp + 1) = \ffh(\pp) + \card\{\ii \mid \ff\inv(\pp) < \ii < \ff\inv(\pp + 1) \ \text{and} \ \ff(\ii) > \pp + 1\},$$
whence $\ffh(\pp + 1) \le \ffh(\pp) + \mm - 1$ and, for there, $\fft(\pp + 1) \ge \fft(\pp) + 2$. Let $\gg:= \ss_\pp \ff$. We find $\gg\inv(\pp) = \ff\inv(\pp + 1)$, $\gg\inv(\pp + 1) = \ff\inv(\pp)$, then $\ggh(\pp) = \ffh(\pp + 1) + 1$ and $\ggh(\pp + 1) = \ffh(\pp)$, because $\ff\inv(\pp)$ contributes to~$\ggh(\pp)$ but not to~$\ffh(\pp + 1)$, and, finally, $\ggt(\pp) = \fft(\pp + 1) - 1$ and $\ggt(\pp + 1) = \fft(\pp)$, with $\ggt(\qq) = \fft(\qq)$ for $\qq \not= \pp, \pp + 1$. So $\ww_\gg$ is the result of applying the rule $\gf{\fft(\pp)}\gf{\fft(\pp + 1)} \to \gf{\fft(\pp + 1) - 1}\gf{\fft(\pp)}$ to~$\ww_\ff$ in position~$\pp$. 

On the other hand, for $\ff\inv(\pp) = \ff\inv(\pp + 1) + \mm$ with $\mm \ge 1$, we find $\ffh(\pp) \le \ffh(\pp + 1) + \mm$, leading to $\fft(\pp) \ge \fft(\pp + 1) + 1$. For $\gg:= \ss_\pp \ff$, we find now, $\ggh(\pp) = \ffh(\pp + 1)$ and $\ggh(\pp + 1) = \ffh(\pp) - 1$, whence $\ggt(\pp) = \fft(\pp + 1)$ and $\ggt(\pp + 1) = \fft(\pp) + 1$, with $\ggt(\qq) = \fft(\qq)$ for $\qq \not= \pp, \pp + 1$. So $\ww_\gg$ is the result of applying the rule $\gf{\fft(\pp)}\gf{\fft(\pp + 1)} \to \gf{\fft(\pp + 1)}\gf{\fft(\pp) + 1}$ to~$\ww_\ff$ in position~$\pp$.

Now, \eqref{E:ExpDtF} implies that the family $\WW:= \{\ww_\ff \mid \ff \in \Sym{\nn - 1}\}$ is closed under~$\eqF$. As the transpositions~$\ss_\ii$ generate~$\Sym{\nn - 1}$, this family~$\WW$ is the $\eqF$-equivalence class of the word~$\ww_{\idf}$, which, by definition, is~$\Dtt\nn$.
\end{proof}

From there, a complete description of simple elements of~$\FFp$ follows:

\begin{prop}\label{P:SimpleF}
For every~$\aa$ in~$\FFp$, the following are equivalent:

\ITEM1 The element~$\aa$ is simple, \ie, $\aa$ left divides some element~$\Dt\nn$;

\ITEM2 The element~$\aa$ is a factor of some element~$\Dt\nn$;

\ITEM3 The normal form of~$\aa$ has the form $\gf{\ii_1} \pdots \gf{\ii_\ell}$ with $\ii_1 > \pdots > \ii_\ell$.

\noindent Moreover, $\aa$ left divides~$\Dt\nn$ if, and only if, $\NF(\aa)$ is $\gf{\ii_1} \pdots \gf{\ii_\ell}$ with $\nn > \ii_1 > \pdots > \ii_\ell$.
\end{prop}

\begin{proof}
By definition, \ITEM1 implies~\ITEM2. Next, assume that $\aa$ is a factor of~$\Dt\nn$, say $\Dt\nn = \aa_1 \aa \aa_2$. Let $\ww_1, \ww, \ww_2$ be the normal forms of~$\aa_1$, $\aa$, and~$\aa_2$, respectively. Then $\ww_1 \ww \ww_2$ is an expression of~$\Dt\nn$, so, by Lemma~\ref{L:ExpDtF}\ITEM2, it is a word~$\ww_\ff$ for some permutation~$\ff$. Moreover, because $\ww$ is~$\EF$-reduced, no rule of~$\EF$ may apply to it: by~\eqref{E:ExpDtF}, this implies that the indices of the generators~$\gf\ii$ in~$\ww$ make a decreasing sequence. So \ITEM2 implies~\ITEM3.

Assume now $\ww = \gf{\ii_1} \pdots \gf{\ii_\ell}$ with $\ii_1 > \pdots > \ii_\ell$. By inserting intermediate letters when $\ii_\pp \ge \ii_{\pp + 1} + 2$, we obtain a word~$\ww'$ that is the normal form of~$\Dt{\ii_1+1}$. Then, repeatedly applying to~$\ww'$ some relations $\gf\ii \gf\jj \to \gf\jj \gf{\ii + 1}$, we push the new letters to the right starting with the last one and finishing with the first one. In this way, one obtains a new expression of~$\Dt{\ii_1+1}$ that begins with~$\ww$. So $\ww$ is the normal form of a prefix of~$\Dt{\ii_1+1}$, hence of a simple element. Si \ITEM3 implies~\ITEM1.

For the last sentence, if $\aa$ left divides~$\Dt\nn$, then so does the first generator of~$\NF(\aa)$: by Lemma~\ref{L:AtomDivF}, the latter cannot be~$\gf\ii$ with~$\ii \ge \nn$. Conversely, the above proof of \ITEM3 $\Rightarrow$~\ITEM1 shows that $\gf{\ii_1} \pdots \gf{\ii_\ell}$ left divides~$\Dt{\ii_1 + 1}$, hence~$\Dt\nn$ for~$\nn > \ii_1$.
\end{proof}

\begin{coro}\label{C:SimpleF}
\ITEM1 For every~$\nn$, the number of left divisors of~$\Dt\nn$ in~$\FFp$ is~$2^{\nn - 1}$.

\ITEM2 Simple elements of~$\FFp$ make a Garside family in~$\FFp$.
\end{coro}

\begin{proof}
\ITEM1 By the last statement in Proposition~\ref{P:SimpleF}, mapping a subset of~$\{1 \wdots \nn{-}\nobreak1\}$ to the decreasing enumeration of the corresponding elements~$\gf\ii$ establishes a one-to-one correspondence between~$\Pw(\{1 \wdots \nn - 1\})$ and the left divisors of~$\Dt\nn$.

\ITEM2 By definition, the family of simple elements in~$\FFp$ is closed under right lcm: the conjunction of~$\aa \dive \Dt\nn$ and $\bb \dive \Dt\pp$ implies that the right lcm of~$\aa$ and~$\bb$ left divides~$\Dt{\max(\nn, \pp)}$. On the other hand, a right divisor of a simple element must be a factor of some~$\Dt\nn$, hence, by Proposition~\ref{P:SimpleF}, it is simple. By~\cite[Coro.~IV.2.29]{Dir}, this implies that simple elements form a Garside family.
\end{proof}

As simple elements form a Garside family, every element of~$\FFp$ admits a unique \emph{greedy decomposition} in terms of simple elements, namely a decomposition~$\aa_1 \wdots \aa_\pp$ with $\aa_1 \wdots \aa_\pp$ simple, $\aa_\pp \not= 1$, and, for each~$\ii$, the entry~$\aa_\ii$ is the maximal simple left divisor of~$\aa_\ii \pdots \aa_\pp$, \cite[Prop.~IV.1.20]{Dir}. In the current case, the greedy decomposition is directly connected with the $\EF$-normal form: $\NF(\aa_1) \wdots \NF(\aa_\pp)$ are  the maximal decreasing factors of~$\NF(\aa)$. For instance, for $\NF(\aa) = \gf4 \gf3 \gf2 \gf3 \gf1 \gf1 \gf2$, the greedy decomposition has four entries, namely $\gf4 \gf3 \gf2$, $\gf3 \gf1$, $\gf1$, and~$\gf2$.

From there, all results involving greedy decompositions are valid in~$\FFp$. However, this Garside structure of~$\FFp$ is mostly trivial, exactly parallel to the case of the free commutative monoid~$\NNNN^{(\infty)}$, where simple elements also correspond to finite subsets of generators. In fact, the relations of~$\PF$ are in essence a shifted version of the commutation rules of a free commutative monoid.

\section{The monoid~$\HHp$}\label{S:Hp}

The previous results are elementary and easy, and we now switch to a combinatorially more intricate and interesting situation, connected with the new hybrid~$\HHp$ between Thompson's monoid~$\FFp$ and Artin's braid monoid~$\Bip$ mentioned in the introduction. Our aim will be to develop the same analysis as in the case of~$\FFp$, namely understanding the structure of simple elements, defined as the left divisors of the right lcms of atoms. To this end, we shall follow the same scheme as in Section~\ref{S:Fp} and use both a normal form associated with a rewrite system (Section~\ref{SS:NFH}) and the reversing transformation associated with the presentation (Section~\ref{SS:RevH}).

\subsection{Presentation and first properties}\label{SS:Hp}

We recall that $\HHp$ is the monoid defined by the explicit presentation called~\eqref{E:PresH} in the introduction
\begin{equation*}
\HHp := \bigg\langle \gh1,\gh2, ... \ \bigg\vert\ 
\begin{matrix}
\gh\jj \gh\ii = \gh\ii \gh{\jj + 1} &\text{for} &\jj \ge \ii + 2\\
\gh\jj \gh\ii \gh\jj = \gh\ii \gh\jj \gh{\ii + 3} &\text{for} &\jj = \ii + 1
\end{matrix}
\ \bigg\rangle^{\!\scriptstyle+},
\end{equation*}
hereafter denoted by~$\PH$. We put $\AH:= \{\gh\ii \mid \ii \ge 1\}$, and write~$\eqp$ for the congruence on~$\AHs$ generated by the relations of~$\PH$. For~$\ww$ a word of~$\AHs$, we write~$\cl\ww$ for the~$\eqp$-class of~$\ww$. The relations of~$\PH$ should appear as a mixture of the Thompson relations (as for length~$2$ relations), and of braid relations (as for length~$3$ relations). We immediately see that $\PH$ is a homogeneous presentation, and we can refer without ambiguity to the length~$\lgg\aa$ of an element~$\aa$ of~$\HHp$, defined to be the common length of all words of~$\AHs$ that represent~$\aa$. We also observe that the relations are invariant under shifting the indices of the~$\gh\ii$s by~$+1$, implying that mapping~$\gh\ii$ to~$\gh{\ii + 1}$ for each~$\ii$ induces a well defined endomorphism of~$\HHp$.

Unlike the case of~$\Bip$, the family of generators occurring in a word is not invariant under~$\eqp$: for instance, $\gh3\gh1$ is equal to~$\gh1\gh4$. However, we can easily construct an upper bound on the indices of the generators possibly occurring in the expressions of an element.

\begin{lemm}\label{L:Ceiling}
Define the \emph{ceiling}~$\plf\ww$ of a nonempty word~$\ww = \gh{\ii_1} \pdots \gh{\ii_\ell}$ of~$\AHs$ by
\begin{equation}
\plf\ww:= \max\{\ii_\pp + \ell - \pp \mid \pp = 1 \wdots \ell \}.
\end{equation}
Then $\plf\ww$ is invariant under~$\eqp$.
\end{lemm}

\begin{proof}
It suffices to consider the case of two words~$\ww, \ww'$ deduced from one another by applying one relation of~$\PH$. For $\ww = \uu \gh{\jj} \gh\ii \vv$ and $\ww' = \uu \gh\ii \gh{\jj + 1} \vv$, with $\jj \ge\nobreak \ii +\nobreak 2$, one finds $\plf\ww = \max(\plf\uu + \lgg\vv + 2, \jj + 1 + \lgg\vv, \plf\vv) = \plf{\ww'}$. Similarly, for $\ww =\nobreak \uu \gh\ii \gh{\ii + 1} \gh{\ii + 3} \vv$ and $\ww' = \uu \gh{\ii + 1} \gh\ii \gh{\ii + 1} \vv$, one obtains $\plf\ww = \max(\plf\uu + \lgg\vv + 3, \ii + 3 + \lgg\vv, \plf\vv) = \plf{\ww'}$.
\end{proof}

For~$\aa$ in~$\HHp$, we write $\plf\aa$ for the common value of~$\plf\ww$ for~$\ww$ representing~$\aa$. A direct application is the following a priori nontrivial result:

\begin{prop}\label{P:WordPbH}
The word problem for~$\PH$ is decidable.
\end{prop}

\begin{proof}
For every word~$\ww$ in~$\AHs$, the $\eqp$-class of~$\ww$ is finite: indeed, $\ww' \eqp \ww$ implies both $\lgg{\ww'} = \lgg\ww$ and $\plf{\ww'} = \plf\ww$, and the number of words~$\ww'$ satisfying these conditions is bounded above by~$\plf\ww^{\lgg\ww}$. Therefore, starting from two words~$\ww, \ww'$, one can decide whether $\ww' \eqp \ww$ holds by saturating~$\{\ww\}$ with respect to the relations of~$\PH$, eventually obtaining in finitely many steps an exhaustive enumeration of the $\eqp$-class of~$\ww$. Then one compares~$\ww'$ with the elements of the list so constructed.
\end{proof}

Another property that directly follows from the presentation is the fact that the monoid~$\FFp$ is a quotient of~$\HHp$:

\begin{prop}\label{P:Proj}
The map~$\pi: \gh\ii \mapsto \gf\ii$ induces a surjective homomorphism from the monoid~$\HHp$ onto the Thompson monoid~$\FFp$.
\end{prop}

\begin{proof}
Let $\pi^*$ be the extension of~$\pi$ into a homomorphism from the free monoid~$\AH^*$ to the monoid~$\FFp$. We claim that $\ww \eqp \ww'$ implies $\pi^*(\ww) = \pi^*(\ww')$. It is enough to check this when $\ww {=} \ww'$ is a relation of~$\PH$. The case of length~$2$ relations is trivial, as the latter are relations of~$\PF$. For length~$3$ relations, we find in~$\FFp$
$$\pi^*(\gh{\ii + 1} \gh\ii \gh{\ii + 1}) = \gf{\ii + 1} \gf\ii \gf{\ii + 1} = \gf\ii \gf{\ii + 2} \gf{\ii + 1} = \gf\ii \gf{\ii + 1} \gf{\ii + 3} = \pi^*(\gh\ii \gh{\ii + 1} \gh{\ii + 3}).$$
So $\pi^*$ induces a homomorphism from~$\HHp$ to~$\FFp$. The latter is surjective since each generator~$\gf\ii$ lies in the image.
\end{proof}

The projection~$\pi$ from~$\HHp$ to~$\FFp$ provided by Proposition~\ref{P:Proj} is not injective: $\gh2\gh1$ and~$\gh1\gh3$ are distinct in~$\HHp$ since no relation of~$\PH$ applies to the corresponding words, but they both project to~$\gf2\gf1$ in~$\FFp$.

\subsection{A normal form on~$\HHp$}\label{SS:NFH}

Like in the case of~$\FFp$, our first method for investigating the monoid~$\HHp$ is to construct a normal form using a rewrite system.

\begin{lemm}\label{L:ConvH}
Let~$\EH$ be the rewrite system on~$\AHs$ defined by the rules
\begin{gather}
\label{E:RSH1}
\gh\ii \gh{\jj + 1} \to \gh{\jj} \gh\ii \text{\quad for $\ii \ge 1$ and $\jj \ge \ii + 2$},\\
\label{E:RSH2}
\gh\ii \gh{\ii + 1}\gh{\ii + 3} \to \gh{\ii + 1}\gh\ii \gh{\ii + 1}\text{\quad for $\ii \ge 1$}.
\end{gather}
Then $\EH$ is convergent.
\end{lemm}

\begin{proof}
As in the case of~$\EF$, we show that $\EH$ is noetherian and locally confluent, and appeal to Newman's diamond lemma. Let $\pi$ denote the homomorphism from~$\AHs$ on~$\AFs$ that maps~$\gh\ii$ to~$\gf\ii$ for every~$\ii$. Then, for every~$\ww$ in~$\AHs$ and every integer~$\mm$,
\begin{equation}\label{E:RSProj}
\ww \toH^\mm \ww' \quad\text{implies}\quad \pi(\ww) \toF^\pp \pi(\ww') \text{\quad for some~$\pp$ satisfying $\mm \le \pp \le 2\mm$}.
\end{equation}
Indeed, up to applying~$\pi$, \eqref{E:RSH1} is a rule of~$\EF$, whereas, for~\eqref{E:RSH2}, we find
$$\pi(\gh\ii \gh{\ii + 1}\gh{\ii + 3}) = \gf\ii \gf{\ii + 1}\gf{\ii + 3} \toF \gf\ii \gf{\ii +2}\gf{\ii + 1} \toF \gf{\ii +1} \gf\ii \gf{\ii + 1} = \pi(\gh{\ii + 1} \gh\ii \gh{\ii + 1}).$$ 
Then an infinite nontrivial sequence of $\EH$-reductions would project to an infinite nontrivial sequence of $\EF$-reductions, so the noetherianity of~$\EF$ implies that of~$\EH$.

We now check local confluence. As in Lemma~\ref{L:ConvF}, it is sufficient to consider the critical cases where two rules overlap. As there are two types of rules, four patterns are possible. Twice using~\eqref{E:RSH1} has already been seen (up to a change of letters) in~\eqref{E:LocConfF}. The remaining three cases then correspond to the confluence diagrams
\begin{equation}\label{E:LocConfH1}
\VR(9,7)\begin{picture}(110,0)(0,6)
\put(0,6){$\gh\ii\gh{\ii+1}\gh{\ii+3}\gh{\ii+4}\gh{\ii+6}$}
\put(40,12){$\gh{\ii+1}\gh\ii\gh{\ii+1}\gh{\ii+4}\gh{\ii+6}$}
\put(40.3,0){$\gh\ii\gh{\ii+1}\gh{\ii+4}\gh{\ii+3}\gh{\ii+4}$}
\put(80,6){$\gh{\ii+2}\gh{\ii+1}\gh{\ii+2}\gh\ii\gh{\ii+1}$\ ,}
\put(30,3){\rotatebox{-30}{\hbox{$\toH$}}}
\put(30,8.5){\rotatebox{30}{\hbox{$\toH$}}}
\put(70,1){\rotatebox{30}{\hbox{$\toH^5$}}}
\put(70,10.5){\rotatebox{-30}{\hbox{$\toH^5$}}}
\end{picture}
\end{equation}
\begin{equation}\label{E:LocConfH2}
\VR(9,7)\begin{picture}(80,0)(0,6)
\put(-3,6){$\gh\ii\gh{\jj + 1} \gh{\jj+2}\gh{\jj+4}$}
\put(32,12){$\gh{\jj}\gh\ii \gh{\jj+2}\gh{\jj+4}$}
\put(30,0){$\gh\ii \gh{\jj+2} \gh{\jj + 1} \gh{\jj+2}$}
\put(63,6){$\gh{\jj + 1} \gh{\jj }\gh{\jj + 1} \gh\ii$}
\put(22,3){\rotatebox{-30}{\hbox{$\toH$}}}
\put(22,8.5){\rotatebox{30}{\hbox{$\toH$}}}
\put(55,1){\rotatebox{30}{\hbox{$\toH^3$}}}
\put(55,10.5){\rotatebox{-30}{\hbox{$\toH^3$}}}
\end{picture}
\text{\quad with $\jj\ge\ii+2$,}
\end{equation}
\begin{equation}\label{E:LocConfH3}
\VR(9,7)\begin{picture}(85,0)(-2,6)
\put(-3,6){$\gh\ii\gh{\ii+1}\gh{\ii+3}\gh{\jj + 1}$}
\put(31,12){$\gh{\ii+1}\gh\ii \gh{\ii+1}\gh{\jj + 1}$}
\put(33,0){$\gh\ii\gh{\ii+1}\gh{\jj}\gh{\ii+3}$}
\put(63,6){$\gh{\jj-2}\gh{\ii+1}\gh\ii \gh{\ii+1}$}
\put(22,3){\rotatebox{-30}{\hbox{$\toH$}}}
\put(22,8.5){\rotatebox{30}{\hbox{$\toH$}}}
\put(55,1){\rotatebox{30}{\hbox{$\toH^3$}}}
\put(55,10.5){\rotatebox{-30}{\hbox{$\toH^3$}}}
\end{picture}
\text{\quad with $\jj\ge\ii+5$,}
\end{equation}
which complete the verification.
\end{proof}

We deduce:

\begin{prop}\label{P:RedH}
$\EH$-reduced words provide a unique normal form for the elements of the monoid~$\HHp$.
\end{prop}

For~$\aa$ in~$\HHp$, we shall denote by~$\NF(\aa)$ the unique $\EH$-reduced word that represents~$\aa$. For~$\ww$ in~$\AHs$, we denote by~$\redH\ww$ the unique $\EH$-reduced word to which $\ww$ is $\EH$-reducible. As in Section~\ref{SS:NFF}, we note that a $\AH$-word is $\EH$-reduced if, and only if, it contains no factor in a list of obstructions, here
\begin{equation}\label{E:Obst}
\LI:= \{\gh\ii \gh\jj \mid \jj \ge \ii + 3\} \cup \{\gh\ii \gh{\ii + 1} \gh{\ii + 3} \mid \ii \ge 1\}.
\end{equation}
This implies that, for every~$\nn$, the family of all $\EF$-reduced words lying in~$\{\gh1 \wdots \gh\nn\}^*$ is a regular language. The above characterization of $\EH$-reduced words implies the following useful properties:

\begin{coro}\label{C:RedH}

\ITEM1 Every factor of an $\EH$-reduced word is $\EH$-reduced.

\ITEM2 A word~$\ww$ is $\EH$-reduced if, and only if, all length~$3$ factors of~$\ww$ are.

\ITEM3 If $\uu\vv$ and $\vv\ww$ are $\EH$-reduced, then $\uu\vv\ww$ is $\EH$-reduced, except for:

\HS5- $\lgg\vv=0$, $\uu=\uu'\gh\ii$, and $\ww=\gh\jj\ww'$ with $\jj\ge\ii+3$,

\HS5- $\lgg\vv=0$, $\uu=\uu'\gh\ii$, and $\ww=\gh{\ii+1}\gh{\ii+3}\ww'$,

\HS5- $\lgg\vv=0$, $\uu=\uu'\gh\ii\gh{\ii+1}$, and $\ww=\gh{\ii+3}\ww'$,

\HS5- $\lgg\vv=1$, $\uu=\uu'\gh{\ii}$, $\vv=\gh{\ii+1}$, and $\ww=\gh{\ii+3}\ww'$.
\end{coro}

\begin{proof}
Points~\ITEM1 and~\ITEM2 directly follow from the characterization of $\EH$-reduced words, and so does the fact that $\uu\vv\ww$ is not $\EH$-reduced if one is in one of the four listed cases. The point is, assuming that $\uu\vv\ww$ is not $\EH$-reduced, to prove that one is necessarily in one of the listed cases. Now the assumption that $\uu\vv\ww$ is not $\EH$-reduced means that at least one rule of~$\EH$ can be applied, and, owing to~\ITEM2, the assumption about~$\uu\vv$ and~$\vv\ww$ requires that $\vv$ has length at most one. Considering the various possibilities yields the four identified cases.
\end{proof}

The next result shows that, if $\ww$ is an $\EH$-reduced word, then the $\EH$-reduced form of~$\ww\gh\ii$ is obtained by pushing~$\gh\ii$ to the left as much as possible:.

\begin{lemm}\label{L:Reduction}
If $\ww$ is $\EH$-reduced, then, for every~$\ii$, we have $\redH{\ww \gh\ii} = \ww_1\gh{\ii -\lgg{\ww_2}} \ww_2$ for some decomposition~$(\ww_1, \ww_2)$ of~$\ww$.
\end{lemm}

\begin{proof}
We use induction on the length of~$\ww$. For~$\ww$ empty, the result is obvious. So assume $\lgg\ww\ge 1$. Then we have $\ww = \ww' \gh\kk$ for some~$\kk$. As $\ww'\gh\kk$ is $\EH$-reduced, the possible rewritings of~$\ww'\gh\kk\gh\ii$ necessarily involve the final letter~$\gh\ii$.
	
For $\ii \le \kk+1$, no rule applies to~$\ww'\gh\kk\gh\ii$, so $\ww\gh\ii$ is $\EH$-reduced, and the result is true for $(\ww_1,\ww_2) :=(\ww,\ew)$.
	
For $\ii \ge \kk+3$, we have $\ww'\gh\kk\gh\ii \toH \ww'\gh{\ii-1}\gh\kk$. Now $\ww'$ is $\EH$-reduced and shorter than~$\ww$. Hence, by induction hypothesis, there exists a decomposition $\ww'=\ww_1'\ww_2'$ satisfying $\redH{\ww'\gh{\ii-1}} = \ww_1'\gh\jj \ww_2'$ with~$\jj = \ii - 1 - \lgg{\ww_2'}$. Now $\ww_1'\gh\jj\ww_2'$ is $\EH$-reduced, and so is $\ww_2'\gh\kk$ as a factor of the $\EH$-reduced word~$\ww_1'\ww_2'\gh\kk$. By Corollary~\ref{C:RedH}, $\ww_1'\gh\jj\ww_2'\gh\kk$ is $\EH$-reduced, as we have $\jj \ge \kk + 2 -\lgg{\ww_2'}$. Hence, $\redH{\ww\gh\ii}$ is $\ww_1'\gh\jj\ww_2'\gh\kk$, and the result is true with $(\ww_1, \ww_2) := (\ww_1',\ww_2'\gh\kk)$.
	
There remains the case $\ii = \kk +2$. For $\ww'=\ew$, we have $\ww\gh\ii=\gh{\ii-2}\gh\ii$, which is $\EH$-reduced, and the result is true for $(\ww_1,\ww_2) :=(\ww,\ew)$. Otherwise, we write $\ww' = \ww''\gh\ell$. For $\ell\ne\ii-3$, we find $\ww\gh\ii = \ww''\gh\ell\gh{\ii-2}\gh\ii$, which is $\EH$-reduced as, by assumption, $\ww''\gh\ell\gh{\ii-2}$ is $\EH$-reduced. So the result is true for $(\ww_1,\ww_2) :=(\ww,\ew)$.
	
Finally, for $\ell= \ii-3$, we have $\ww'' \gh{\ii-3} \gh{\ii-2} \gh\ii \eqp \ww'' \gh{\ii-2} \gh{\ii-3} \gh{\ii-2}$. As $\ww''$ is $\EH$-reduced and shorter than~$\ww$, the induction hypothesis gives a decomposition $\ww''=\ww_1''\ww_2''$ satisfying $\redH{\ww''\gh{\ii-2}} = \ww_1''\gh\jj\ww_2''$ with $\jj = \ii - 2 - \lgg{\ww_2''}$. It remains to show that $\ww_1''\gh\jj\ww_2''\gh{\ii-3}\gh{\ii-2}$ is $\EH$-reduced. Now $\ww_1''\gh\jj\ww_2''$ is $\EH$-reduced, and so is $\ww_2''\gh{\ii-3}\gh{\ii-2}$, as a factor of the $\EH$-reduced word $\ww_1'' \ww_2'' \gh{\ii-3} \gh{\ii-2}$. By Corollary~\ref{C:RedH}, $\ww_1''\gh\jj\ww_2''\gh{\ii-3}\gh{\ii-2}$ is $\EH$-reduced, and $\jj = \ii - 2 -\lgg{\ww_2''}$ holds. Therefore, $\redH{\ww\gh\ii}$ is $\ww_1'' \gh\jj \ww_2'' \gh{\ii-3} \gh{\ii-2}$, and the result is true for $(\ww_1,\ww_2) :=(\ww_1'',\ww_2''\gh{\ii-3}\gh{\ii-2})$. 
\end{proof}

We shall now apply the normal form provided by~$\EH$ to studying right cancellativity in~$\HHp$. At this point, we shall not obtain a complete answer, but only a (surprising) connection between left and right cancellativity.

\begin{prop}\label{P:RCanc}
If $\HHp$ is left cancellative, then it is right cancellative as well.
\end{prop}

\begin{proof}
We assume that $\HHp$ is left cancellative, and aim at proving that any equality $\aa\gh\ii = \bb \gh\ii$ implies $\aa = \bb$. So assume $\aa\gh\ii = \bb\gh\ii$. Let $\uu := \NF(\aa)$ and $\vv:= \NF(\bb)$. By Lemma~\ref{L:Reduction}, there exist $\uu_1, \uu_2, \vv_1, \vv_2$ satisfying
\begin{gather}\label{E:NFSimplD1}
\uu = \uu_1\uu_2 \quad\text{and}\quad\redH{\uu\gh\ii} = \uu_1\gh\jj\uu_2 \quad \text{with $\jj=\ii-\lgg{\uu_2}$},\\
\label{E:NFSimplD2}
\vv = \vv_1\vv_2 \quad\text{and}\quad\redH{\vv\gh\ii} = \vv_1\gh\kk\vv_2 \quad \text{with $\kk=\ii-\lgg{\vv_2}$}.
\end{gather}
By assumption, we have $\aa\gh\ii = \bb\gh\ii$, hence $\redH{\uu\gh\ii} = \redH{\vv\gh\ii}$, and, from there, $\uu_1\gh\jj\uu_2 =\nobreak \vv_1\gh\kk\vv_2$. We consider the various possible cases.

Assume first $\jj\ne\kk$, say $\jj<\kk$. By~\eqref{E:NFSimplD1} and~\eqref{E:NFSimplD2}, we obtain
$$\lgg{\uu_2} =\nobreak \ii - \jj > \ii - \kk = \lgg{\vv_2}, \quad\text{whence}\quad\lgg{\uu_1} < \lgg{\vv_1}.$$
So $\uu_1$ is a proper prefix of~$\vv_1$, and $\vv_2$ is a proper suffix of~$\uu_2$. As $\uu_1$ is a proper prefix of~$\vv_1$, the word~$\uu_1\gh\jj$ is a prefix of~$\vv_1$, and there exists~$\ww$ satisfying $\vv_1=\uu_1\gh\jj \ww$. We find $\uu_1\gh\jj\uu_2 = \uu_1 \gh \jj \ww \gh\kk\vv_2$, hence $\uu_2 = \ww \gh\kk\vv_2$. Therefore, we have
$$\uu = \uu_1 \ww \gh\kk\vv_2, \quad \vv=\uu_1\gh\jj \ww \vv_2, \quad \text{and} \quad \redH{\uu\gh\ii} = \redH{\vv\gh\ii} = \uu_1\gh\jj \ww \gh\kk\vv_2.$$
The equality $\redH{\vv\gh\ii} = \uu_1\gh\jj \ww \gh\kk\vv_2$ implies $\vv\gh\ii \eqp\uu_1\gh\jj \ww \gh\kk\vv_2$, that is, $\uu_1\gh\jj \ww \vv_2\gh\ii \eqp\uu_1\gh\jj \ww \gh\kk\vv_2$. As~$\HHp$ is left cancellative, left cancelling~$\uu_1\gh\jj \ww $ yields $\vv_2\gh\ii\eqp\gh\kk\vv_2$. By assumption, $\uu_1$ is $\EH$-reduced, hence, by Corollary~\ref{C:RedH}\ITEM1, so is its suffix~$\gh\kk\vv_2$. The equivalence $\vv_2\gh\ii\eqp\gh\kk\vv_2$ implies $\redH{\vv_2\gh\ii} = \gh\kk\vv_2$, hence $\vv_2\gh\ii \TOH \gh\kk\vv_2$, whence $\uu\gh\ii = \uu_1 \ww \gh\kk\vv_2\gh\ii \TOH \uu_1 \ww \gh\kk\gh\kk\vv_2$. Now, as a prefix of~$\uu$, the word $\uu_1 \ww \gh\kk$ is $\EH$-reduced, whereas $\gh\kk\gh\kk$ is $\EH$-reduced by definition. By Corollary~\ref{C:RedH}\ITEM3, $\uu_1 \ww \gh\kk\gh\kk$ is reduced. On the other hand, as a suffix of~$\uu_1$, the word $\gh\kk\vv_2$ is $\EH$-reduced, so, by Corollary~\ref{C:RedH}\ITEM3 again, $\gh\kk\gh\kk\vv_2$ is $\EH$-reduced. Finally, $\uu_1 \ww \gh\kk\gh\kk$ and $\gh\kk\gh\kk \ww $ are $\EH$-reduced, hence, by Corollary~\ref{C:RedH}\ITEM2, $\uu_1 \ww \gh\kk\gh\kk\vv_2$ is $\EH$-reduced. So the two words $\uu_1 \gh\jj \ww \gh\kk \vv_2$ and $\uu_1 \ww \gh\kk\gh\kk \vv_2 $ are $\EH$-reduced, both equivalent to~$\uu\gh\ii$. Hence they must coincide: $\uu_1 \ww \gh\kk\gh\kk \vv_2 =\uu_1\gh\jj \ww \gh\kk \vv_2 $ holds. Deleting the prefix~$\uu_1$ and the suffix~$\gh\kk \vv_2$ on both sides, we deduce
\begin{equation}\label{E:FNSimpl}
 \ww \gh\kk=\gh\jj \ww.
\end{equation}
An induction on~$\lgg\ww$ shows that the word equality (not equivalence)~\eqref{E:FNSimpl} is possible only for~$\jj = \kk$: for $\lgg\ww \ge 2$, a word~$\ww$ satisfying~\eqref{E:FNSimpl} must begin with~$\gh\jj$ and finish with~$\gh\kk$, leading to~$\ww = \gh\jj \ww' \gh\kk$ with~$\ww'$ satisfying~$\ww' \gh\kk=\gh\jj \ww'$. But this contradicts the assumption~$\jj \not= \kk$.

So, the only possibility is $\jj = \kk$. Then \eqref{E:NFSimplD1} and~\eqref{E:NFSimplD2} imply $\lgg{\uu_2} =\nobreak \lgg{\vv_2}$, whence $\uu_2 = \vv_2$, and, from there, $\uu_1=\vv_1$ and $\uu = \vv$, implying $\aa = \bb$. 
\end{proof}

Another application of the normal form in~$\HHp$ is a solution for the word problem of the presentation~$\PH$ that is much more efficient than the ``stupid'' solution of Proposition~\ref{P:WordPbH}: two words~$\ww, \ww'$ represent the same element in~$\HHp$ if, and only if, $\redH\ww$ and~$\redH{\ww'}$ coincide. It is easy to see that, from a word of length~$\ell$, at most $\binom{\ell}{2}$ rules can be applied, leading to a solution for the word problem whose overall complexity is quadratic in~$\ell$. We do not go into details here.

\subsection{Reversing for~$\HHp$}\label{SS:RevH}

Continuing as in Section~\ref{SS:RevF} for~$\FFp$, we now investigate the (right) reversing relation associated with the presentation~$\PH$ of~$\HHp$, in view of possibly establishing that it is left cancellative and admits right lcms. 

For applying Lemma~\ref{L:Rev}, the first step is to check Condition~\eqref{E:CompatS}.

\begin{lemm}\label{L:CompatH}
For every generator~$\gh\ii$ and for every relation~$\ww = \ww'$ of~$\PH$, Condition~\eqref{E:CompatS} is satisfied: for every $\PH$-grid from~$(\gh\ii, \ww)$, there is an equivalent grid from~$(\gh\ii, \ww')$, and vice versa.
\end{lemm}

\begin{proof}
Because there exists exactly one relation of the form $\gh\ii ... = \gh\jj...$ in~$\PH$ for all~$\ii, \jj$, a $\PH$-grid is unique when it exists, so we only have to check that the involved grids either exist and are equivalent, or they do not exist. As in the case of~$\FFp$, there are infinitely many generators and relations, but only finitely many patterns occur, according to the relative positions of the indices of the involved generators. In the case of~$\gh\ii$ and the relation $\gh\jj \gh{\jj + 1} \gh{\jj + 3} = \gh{\jj + 1} \gh\jj \gh{\jj + 1}$, the only cases that do not just result in shifting the indices are $\jj = \ii + 1$ and $\jj = \ii - 2$, for which we find (for readability, we draw the diagrams for $\ii = 1, \jj = 2$, and for $\ii = 3, \jj = 1$)
\begin{equation}\label{E:Case4}
\VR(14,11)\begin{picture}(40,0)(0,6)
\pcline{->}(1,16)(19,16)\taput{$\gh2$}
\pcline{->}(21,16)(29,16)\taput{$\gh3$}
\pcline{->}(31,16)(39,16)\taput{$\gh5$}
\pcline{->}(21,8)(29,8)\taput{$\gh4$}
\pcline{->}(31,8)(39,8)\taput{$\ \gh6$}
\pcline{->}(1,0)(9,0)\tbput{$\gh2$}
\pcline{->}(11,0)(19,0)\tbput{$\gh4$}
\pcline{->}(21,0)(29,0)\tbput{$\gh5$}
\pcline{->}(31,0)(39,0)\tbput{$\gh7$}
\pcline{->}(0,15)(0,1)\tlput{$\gh1$}
\pcline{->}(20,15)(20,9)\tlput{$\gh1$}
\pcline{->}(20,7)(20,1)\tlput{$\gh2$}
\pcline{->}(30,15)(30,9)\trput{$\gh1$}
\pcline{->}(30,7)(30,1)\trput{$\gh2$}
\pcline{->}(40,15)(40,9)\trput{$\gh1$}
\pcline{->}(40,7)(40,1)\trput{$\gh2$}
\end{picture}
\hspace{20mm}
\begin{picture}(40,0)(0,6)
\pcline{->}(1,16)(9,16)\taput{$\gh3$}
\pcline{->}(11,16)(29,16)\taput{$\gh2$}
\pcline{->}(31,16)(39,16)\taput{$\gh3$}
\pcline{->}(31,8)(39,8)\taput{$\gh4$}
\pcline{->}(1,0)(9,0)\tbput{$\gh4$}
\pcline{->}(11,0)(19,0)\tbput{$\gh2$}
\pcline{->}(21,0)(29,0)\tbput{$\gh4$}
\pcline{->}(31,0)(39,0)\tbput{$\gh5$}
\pcline{->}(0,15)(0,1)\tlput{$\gh1$}
\pcline{->}(10,15)(10,1)\tlput{$\gh1$}
\pcline{->}(30,15)(30,9)\tlput{$\gh1$}
\pcline{->}(30,7)(30,1)\tlput{$\gh2$}
\pcline{->}(40,15)(40,9)\trput{$\gh1$}
\pcline{->}(40,7)(40,1)\trput{$\gh2$}
\end{picture}
\end{equation}
and we check $\gh2\gh4\gh5\gh7 \eqp \gh2\gh5\gh4\gh5 \eqp \gh4\gh2\gh4\gh5$, so the grids are equivalent, and
\begin{equation}\label{E:Case4}
\VR(15,11)\begin{picture}(40,0)(0,6)
\pcline{->}(1,16)(9,16)\taput{$\gh1$}
\pcline{->}(11,16)(19,16)\taput{$\gh2$}
\pcline{->}(21,16)(39,16)\taput{$\gh4$}
\pcline{->}(1,0)(9,0)\tbput{$\gh1$}
\pcline{->}(11,0)(19,0)\tbput{$\gh2$}
\pcline{->}(21,0)(29,0)\tbput{$\gh4$}
\pcline{->}(31,0)(39,0)\tbput{$\gh5$}
\pcline{->}(0,15)(0,1)\tlput{$\gh3$}
\pcline{->}(10,15)(10,1)\tlput{$\gh4$}
\pcline{->}(20,15)(20,1)\trput{$\gh5$}
\pcline{->}(40,15)(40,9)\trput{$\gh5$}
\pcline{->}(40,7)(40,1)\trput{$\gh7$}
\end{picture}
\hspace{20mm}
\begin{picture}(40,0)(0,6)
\pcline{->}(1,16)(19,16)\taput{$\gh2$}
\pcline{->}(21,16)(29,16)\taput{$\gh1$}
\pcline{->}(31,16)(39,16)\taput{$\gh2$}
\pcline{->}(21,8)(29,8)\taput{$\gh1$}
\pcline{->}(31,8)(39,8)\taput{$\ \gh2$}
\pcline{->}(1,0)(9,0)\tbput{$\gh2$}
\pcline{->}(11,0)(19,0)\tbput{$\gh3$}
\pcline{->}(21,0)(29,0)\tbput{$\gh1$}
\pcline{->}(31,0)(39,0)\tbput{$\gh2$}
\pcline{->}(0,15)(0,1)\tlput{$\gh3$}
\pcline{->}(20,15)(20,9)\tlput{$\gh3$}
\pcline{->}(20,7)(20,1)\tlput{$\gh5$}
\pcline{->}(30,15)(30,9)\trput{$\gh4$}
\pcline{->}(30,7)(30,1)\trput{$\gh6$}
\pcline{->}(40,15)(40,9)\trput{$\gh5$}
\pcline{->}(40,7)(40,1)\trput{$\gh7$}
\end{picture}
\end{equation}
and we check $\gh1\gh2\gh4\gh5 \eqp \gh2\gh1\gh2\gh5 \eqp \gh2\gh1\gh4\gh2 \eqp \gh2\gh3\gh1\gh2$, so the grids are equivalent. Similarly, in the case of~$\gh\ii$ and a relation $\gh\jj \gh{\kk + 1} = \gh{\kk} \gh\jj$ with $\kk \ge \jj + 2$, the only nontrivial case is for $\ii = \jj + 1$ and $\kk = \jj + 2$, where we find (here for $\ii = 2$, $\jj = 1$, and $\kk = 3$)
\begin{equation}\label{E:Case4}
\VR(18,15)\begin{picture}(40,0)(0,10)
\pcline{->}(1,24)(19,24)\taput{$\gh1$}
\pcline{->}(21,24)(39,24)\taput{$\gh4$}
\pcline{->}(21,16)(39,16)\taput{$\gh5$}
\pcline{->}(1,0)(9,0)\tbput{$\gh1$}
\pcline{->}(11,0)(19,0)\tbput{$\gh2$}
\pcline{->}(21,0)(29,0)\tbput{$\gh5$}
\pcline{->}(31,0)(39,0)\tbput{$\gh7$}
\pcline{->}(0,23)(0,1)\tlput{$\gh2$}
\pcline{->}(20,23)(20,17)\tlput{$\gh2$}
\pcline{->}(20,15)(20,1)\tlput{$\gh4$}
\pcline{->}(40,23)(40,17)\trput{$\gh2$}
\pcline{->}(40,15)(40,9)\trput{$\gh4$}
\pcline{->}(40,7)(40,1)\trput{$\gh5$}
\end{picture}
\hspace{20mm}
\begin{picture}(40,0)(0,10)
\pcline{->}(1,24)(19,24)\taput{$\gh3$}
\pcline{->}(21,24)(39,24)\taput{$\gh1$}
\pcline{->}(21,8)(29,8)\taput{$\gh1$}
\pcline{->}(31,8)(39,8)\taput{$\ \gh2$}
\pcline{->}(1,0)(9,0)\tbput{$\gh3$}
\pcline{->}(11,0)(19,0)\tbput{$\gh5$}
\pcline{->}(21,0)(29,0)\tbput{$\gh1$}
\pcline{->}(31,0)(39,0)\tbput{$\gh2$}
\pcline{->}(0,23)(0,1)\tlput{$\gh2$}
\pcline{->}(20,23)(20,9)\tlput{$\gh2$}
\pcline{->}(20,7)(20,1)\tlput{$\gh3$}
\pcline{->}(30,7)(30,1)\trput{$\gh4$}
\pcline{->}(40,23)(40,17)\trput{$\gh2$}
\pcline{->}(40,15)(40,9)\trput{$\gh4$}
\pcline{->}(40,7)(40,1)\trput{$\gh5$}
\end{picture}
\end{equation}
and we check $\gh1\gh2\gh5\gh7 \eqp \gh1\gh4\gh2\gh7 \eqp \gh3\gh1\gh2\gh7 \eqp \gh3\gh1\gh6\gh2 \eqp \gh3\gh5\gh1\gh2$, so the grids are equivalent. 
\end{proof}

Lemma~\ref{L:CompatH} shows that the presentation~$\PH$, which is homogeneous, is eligible for Lemma~\ref{L:Rev}. So, in particular, for all~$\uu, \vv$ in~$\AHs$, we have the equivalence
\begin{equation}\label{E:EqRev}
\uu \eqp \vv \quad \Longleftrightarrow \quad (\uu, \vv) \rev (\ew, \ew),
\end{equation}
that is, $\uu$ and $\vv$ represent the same element of~$\HHp$ if, and only if there exists a $\PH$-grid from~$(\uu, \vv)$ to~$(\ew, \ew)$. As an application, we deduce

\begin{prop}\label{P:Canc}
The monoid~$\HHp$ is left and right cancellative.
\end{prop}

\begin{proof}
Lemma~\ref{L:Rev}\ITEM2 implies that $\HHp$ is left cancellative. By Proposition~\ref{P:RCanc}, this implies that $\HHp$ is right cancellative as well.
\end{proof}

As for common right multiples, owing to the fact that a $\PH$-grid with a given source is unique when it exists, Lemma~\ref{L:Rev}\ITEM3 directly implies:

\begin{prop}\label{P:CommMult}
Two elements~$\cl\uu, \cl\vv$ of~$\HHp$ admit a common right multiple in~$\HHp$ if, and only if, there exists a $\PH$-grid from~$(\uu, \vv)$; in this case, $\cl\uu$ and~$\cl\vv$ admit a right lcm, represented by~$\uu\vv_1$ with $(\uu_1, \vv_1)$ the output of the grid from~$(\uu, \vv)$.
\end{prop}

Proposition~\ref{P:CommMult} is optimal: there exist elements of~$\HHp$ without a common right multiple, typically~$\gh2$ and~$\gh1\gh3$. Indeed, if we try to construct a $\PH$-grid from~$(\gh2, \gh1\gh3)$, we must start with
\begin{equation*}
\VR(12,10)\begin{picture}(40,0)(0,7)
\pcline{->}(1,16)(19,16)\taput{$\gh1$}
\pcline{->}(21,16)(39,16)\taput{$\gh3$}
\pcline{->}(1,0)(9,0)\tbput{$\gh1$}
\pcline{->}(11,0)(19,0)\tbput{$\gh2$}
\pcline{->}(0,15)(0,1)\tlput{$\gh2$}
\pcline{->}(20,15)(20,9)\tlput{$\gh2$}
\pcline{->}(20,7)(20,1)\tlput{$\gh4$}
\end{picture}
\end{equation*}
and the process cannot terminate, since the pending pattern $(\gh2\gh4, \gh3)$ is, up to a symmetry, the image of $(\gh2, \gh1\gh3)$ under shifting the indices. By Proposition~\ref{P:CommMult}, this is enough to conclude that $\gh2$ and~$\gh1\gh3$ admit no common right multiple in~$\HHp$.

\begin{rema}
In the above example, the non-existence of a common right multiple is easily established, as constructing a $\PH$-grid enters an explicit non-terminating loop. The general question of the existence of a $\PH$-grid from a given pair of words is a priori difficult, and it is not clear whether it is algorithmically decidable. In fact, it is, but this is nontrivial. The method consists in identifying an explicit family of words~$\AHp$ that is closed under reversing, in the sense that, if $\uu$ and~$\vv$ belong to~$\AHp$ and $(\uu, \vv) \rev (\uu_1, \vv_1)$ holds, then $\uu_1$ and~$\vv_1$ lie in~$\AHp$. Then, one easily shows that, if the existence of common multiples can be decided for pairs of words of~$\AHp$, it can decided for arbitrary pairs of words. So the question is to find a convenient family~$\AHp$ and analyze the existence of $\PH$-grids for words of~$\AHp$. In the case of~$\PH$, this program is successfully completed in~\cite{Tes} for $\AHp:= \AH_1 \cup \AH_2 \cup \{\ew\}$ with
\begin{gather}\label{E:TMClosH1}
\AH_1 := \{\gh\ii\gh{\ii+2}\gh{\ii+4}\pdots\gh{\ii+2\kk} \mid \ii\ge 1, \kk \ge 0\}, \\
\label{E:TMClosH2}
\AH_2 := \{\gh\ii\gh{\ii+2}\gh{\ii+4}\pdots\gh{\ii+2\kk} \gh{\ii+2\kk+1} \mid \ii\ge 1, \kk \ge 0\},
\end{gather} 
Therefore the existence of common right multiples in~$\HHp$ is decidable.
\end{rema}

Trying to apply the same approach to studying right cancellativity and left multiples in~$\HHp$ fails. Indeed, one easily checks that the left diagram below is a legitimate ``left $\PH$-grid'' from~$(\gh6, \gh2 \gh1 \gh2)$, whereas the right diagram shows that constructing an equivalent left $\PH$-grid from~$(\gh6, \gh1 \gh2 \gh4)$ fails, since there is no relation $... \gh2 = ... \gh3$ in~$\PH$: 
\begin{equation}
\VR(15,11)\begin{picture}(40,0)(0,6)
\pcline{->}(1,16)(9,16)\taput{$\gh1$}
\pcline{->}(11,16)(19,16)\taput{$\gh2$}
\pcline{->}(21,16)(29,16)\taput{$\gh1$}
\pcline{->}(31,16)(39,16)\taput{$\gh2$}
\pcline{->}(1,0)(19,0)\tbput{$\gh2$}
\pcline{->}(21,0)(29,0)\tbput{$\gh1$}
\pcline{->}(31,0)(39,0)\tbput{$\gh2$}
\pcline{->}(0,15)(0,9)\tlput{$\gh2$}
\pcline{->}(0,7)(0,1)\tlput{$\gh1$}
\pcline{->}(20,15)(20,1)\tlput{$\gh4$}
\pcline{->}(30,15)(30,1)\trput{$\gh5$}
\pcline{->}(40,15)(40,1)\trput{$\gh6$}
\end{picture}
\hspace{15mm}
\begin{picture}(40,0)(0,6)
\pcline{->}(21,16)(29,16)\taput{$\gh3$}
\pcline{->}(31,16)(39,16)\taput{$\gh4$}
\pcline{->}(1,0)(9,0)\tbput{$\gh1$}
\pcline{->}(11,0)(19,0)\tbput{$\gh2$}
\pcline{->}(21,0)(39,0)\tbput{$\gh4$}
\pcline{->}(20,15)(20,9)\tlput{$\gh4$}
\pcline{->}(20,7)(20,1)\tlput{$\gh3$}
\pcline{->}(40,15)(40,1)\trput{$\gh6$}
\end{picture}
\end{equation}
So the counterpart of Condition~\eqref{E:CompatS} fails for~$\gh6$ and the relation $\gh1\gh2 \gh4 = \gh2 \gh1 \gh2$, and the counterpart of Lemma~\ref{L:Rev} cannot be appealed to. This says nothing about~$\HHp$, but just leaves the questions open. As for right cancellativity, the question was solved using the normal form of Section~\ref{SS:NFH}. As for left multiples, nothing is clear, in particular about the possible existence of left lcms. However, what is clear is that, for instance, $\gh\ii$ and~$\gh{\ii + 1}$ admit no common left multiple, since their projections in~$\FFp$ do not.

\section{Garside combinatorics for~$\HHp$}\label{S:Delta}

We now show that the monoid~$\HHp$ admits interesting combinatorial properties similar to those of~$\FFp$, in connection with distinguished elements~$\Dt\nn$ defined as right lcms of atoms, and with the left divisors of the latter, called simple elements. So our goal is to establish for~$\HHp$ results similar to those of Section~\ref{SS:GarF}. We shall see that the results are indeed similar, but with more delicate and interesting proofs.

\subsection{The elements~$\Dt\nn$ and their divisors}\label{SS:Delta}

The atoms of the monoid~$\HHp$ are the elements~$\gh\ii$ with $\ii \ge 1$. On the shape of the braid and Thompson cases, we shall introduce for every~$\nn$ a distinguished element of~$\HHp$, again denoted~$\Dt\nn$, which is the right lcm of the first $\nn - 1$~atoms. As in Section~\ref{S:Fp}, it will be convenient to start from explicit word representatives. Moreover, in view of subsequent computations, we shall simultaneously introduce, for every~$\nn$, an element~$\Dt{\nn + 0.5}$ that is intermediate between~$\Dt{\nn}$ and~$\Dt{\nn + 1}$.

\begin{defi}\label{D:Delta}	
We put $\Dtt1 = \Dtt{1.5} := \ew$, $\Dtt2 := \gh1$, and, for $\nn \ge 2$, 
\begin{equation}\label{E:DeltaInd}
\Dtt\nn :=\Dtt{\nn - 1}\gh{3\nn-7} \gh{3\nn-5}\quad \text{and}\quad \Dtt{\nn + 0.5}:= \Dtt\nn \gh{3\nn - 4}.
\end{equation}
We denote by~$\Dt\nn$ (\resp~$\Dt{\nn + 0.5}$) the class of~$\Dtt\nn$ (\resp~$\Dtt{\nn + 0.5}$) in~$\HHp$.
\end{defi}

So, by definition, the word~$\Dtt\nn$ (\resp $\Dtt{\nn + 0.5}$) is the increasing enumeration from~$1$ to~$3\nn - 5$ (\resp $3\nn - 4$) of all generators~$\gh\ii$ with $\ii \not= 0 \mod3$. We immediately obtain for every~$\nn \ge 2$
\begin{equation}\label{E:Demi2}
\Dt\nn \dive \Dt{\nn + 0.5} \dive \Dt{\nn + 1},
\end{equation}
where we recall~$\dive$ denotes the left divisibility relation. 
For~$\nn \ge 3$, the word~$\Dtt\nn$ is not $\EH$-reduced; an easy induction gives the values 
\begin{gather}
\label{E:FNDt}
\NF(\Dt\nn) = \gh{\nn - 1} \cdot \gh{\nn - 2}\gh{\nn - 1} \cdot \gh{\nn - 3}\gh{\nn - 2}\cdot \pdots \cdot \gh2 \gh3 \cdot \gh1 \gh2,\\
\label{E:FNDthalf}
\NF(\Dt{\nn+0.5}) = \gh{\nn - 1}\gh\nn \cdot \gh{\nn - 2}\gh{\nn - 1} \cdot \gh{\nn - 3}\gh{\nn - 2}\cdot \pdots \cdot \gh2 \gh3 \cdot \gh1 \gh2.
\end{gather}
Note that \eqref{E:FNDthalf} implies that $\Dt{\nn + 0.5}$, which left divides~$\Dt{\nn + 1.5}$ by~\eqref{E:Demi2}, also right divides it. It also implies, for $\nn \ge 2$, the equality
\begin{equation}\label{E:Delta1}
\Dt\nn = \gh{\nn-1} \Dt{\nn - 0.5}.
\end{equation} 

The first step in studying the element~$\Dt\nn$ is to establish that it is indeed the right lcm of the expected atoms. 

\begin{lemm}\label{L:AtomDiv}
For every~$\nn \ge 2$, the element~$\Dt\nn$ is the right lcm of~$\gh1 \wdots \gh{\nn-1}$. No element~$\gh\ii$ with~$\ii \ge \nn$ left divides either~$\Dt\nn$ or~$\Dt{\nn + 0.5}$.
\end{lemm}

\begin{proof}
We first prove using induction on~$\nn \ge 2$ that $\Dt\nn$ is the right lcm of~$\gh1 \wdots \gh{\nn - 1}$. The result is trivial for $\nn \ge 2$, so assume $\nn \ge 3$. A direct computation gives
$$(\gh{\nn - 1}, \Dtt{\nn - 1}) \rev (\gh{3\nn - 7} \gh{3\nn - 5}, \Dtt{\nn - 0.5}).$$
By Lemma~\ref{L:Rev}\ITEM3, this implies that $\Dtt\nn$ represents the right lcm of~$\gh{\nn - 1}$ and~$\Dt{\nn - 1}$. By induction hypothesis, $\Dt{\nn - 1}$ is the right lcm of~$\gh1 \wdots \gh{\nn - 2}$, so $\Dt\nn$ is the right lcm of~$\gh1 \wdots \gh{\nn - 1}$. On the other hand, for $\ii \ge \nn$, similar computations give
$$(\gh\ii, \Dtt{\nn + 0.5}) \rev (\gh{\ii + 2\nn - 2}, \Dtt{\nn + 0.5}),$$
showing that the right lcm of~$\gh\nn$ and~$\Dt{\nn + 0.5}$ exists but is not~$\Dt{\nn + 0.5}$, so $\gh\ii$ does not left divide~$\Dt{\nn + 0.5}$. Hence, by~\eqref{E:Demi2}, $\gh\ii$ does not left divide~$\Dt\nn$ either. 
\end{proof}

We already mentioned that $\gh1$ and~$\gh2$ admit no common left multiple in~$\HHp$: it follows that $\Dt\nn$ cannot be a left lcm for~$\gh1 \wdots \gh{\nn - 1}$.

\begin{defi}\label{D:Simple}
An element~$\aa$ of~$\HHp$ is called \emph{simple} if $\aa \dive \Dt\nn$ holds for some~$\nn$; in this case, the least such~$\nn$ is called the \emph{index} of~$\aa$, denoted by~$\ind\aa$. For~$\nn \ge 1$ and~$\ell \ge 0$, we put
\begin{equation*}
\SPP\nn\ell:= \{\aa \in \HHp \mid \aa \dive \Dt\nn \ \text{and}\ \lgg\aa = \ell\}, \quad \text{and}\quad 
\SP\nn:= \bigcup\nolimits_{\ell \ge 0}\SPP\nn\ell.
\end{equation*}
\end{defi}

For instance, $\SP3$ is the family of all left divisors of~$\Dt3$; one easily checks that it consists of the six elements~$1$, $\gh1$, $\gh2$, $\gh1\gh2$, $\gh2\gh1$, and~$\Dt3$. On the other hand, Lemma~\ref{L:AtomDiv} implies that $\SPP\nn1$ is equal to~$\{\gh1 \wdots \gh{\nn - 1}\}$ for every~$\nn$. 

As $\Dt\nn$ left divides~$\Dt{\nn + 1}$ for every~$\nn$, if $\aa$ is simple, the values of~$\nn$ satisfying $\aa \dive \Dt\nn$ make an interval~$[\pp, \infty[$, and $\ind\aa$ is the number~$\pp$. Thus, $\aa \dive \Dt\nn$ is equivalent to~$\ind\aa \le \nn$, and $\ind\aa = \nn$ is equivalent to the conjunction of~$\aa \dive \Dt\nn$ and~$\aa \not\dive \Dt{\nn - 1}$.

We shall subsequently need an upper bound for the ceiling of simple elements (as introduced in Lemma~\ref{L:Ceiling}). This can be deduced from Lemma~\ref{L:AtomDiv}.

\begin{lemm}\label{L:PlafSimpl}
For $\aa \dive \Dt{\nn + 0.5}$ in~$\HHp$ with $\nn \ge 2$, one has
\begin{equation}\label{E:PlafSimpl1}
\plf\aa \le \nn+ \lgg\aa -2.
\end{equation}
\end{lemm}

\begin{proof}
Use induction on $\lgg\aa \ge1$. For $\lgg\aa = 1$, Lemma~\ref{L:AtomDiv} implies $\aa = \gh\ii$ with $\ii \le \nn-1$, whence $\plf\aa = \ii \le \nn-1 = \nn + \lgg\aa -2$.
	
Assume now $\lgg\aa \ge 2$, and write $\aa = \bb\gh\ii$. Then $\bb \dive \Dt{\nn + 0.5}$ holds as well, and the induction hypothesis implies $\plf{\bb}\le\nn+\lgg\aa-3$. Write $\bb\gh\ii\cc = \Dt{\nn + 0.5}$. By definition, the contribution of~$\gh\ii$ to the ceiling of~$\Dt{\nn + 0.5}$ is $\ii + \lgg\cc$, \ie, $\ii+2\nn-\lgg\aa-2$. The explicit definition of~\eqref{E:DeltaInd} gives $\plf{\Dt{\nn + 0.5}} = 3\nn - 4$, and we deduce $\ii \le \nn +\lgg\aa-2$, whence, finally, $\plf\aa = \plf{\bb \gh\ii} = \max(\plf{\bb} + 1, \ii) \le \nn+\lgg\aa-2$.
\end{proof}

\subsection{Expressions of~$\Dt\nn$ and~$\Dt{\nn + 0.5}$}\label{SS:Prepa}

Our aim is to analyze simple elements of~$\HHp$ precisely. To this end, it will be crucial to control the various expressions of the elements~$\Dt\nn$ and~$\Dt{\nn + 0.5}$, and we establish here technical results in this direction. Obtaining a complete description as in Lemma~\ref{L:ExpDtF} seems hopeless, but it will be sufficient to connect the expressions of~$\Dt\nn$ with those of~$\Dt{\nn - 0.5}$ (\ie, of~$\Dt{\nn-1}\gh{3\nn-7}$). A similar connection will be stated between the expressions of~$\Dt{\nn - 0.5}$ and~$\Dt{\nn - 1.5}$ in Lemma~\ref{L:ExpDnhalf} below.

\begin{lemm}\label{L:ExpDn}
For $\nn\ge 2$, every expression of~$\Dt{\nn}$ has the form $\ww_1 \gh\kk \ww_2$ with $\ww_1\ww_2 \eqp\nobreak \Dtt{\nn - 0.5}$ and $\kk = \nn+\lgg{\ww_1}-1$; in this case, $\ww_1 \gh\kk \eqp \gh{\nn-1}\ww_1$ holds.
\end{lemm}

\begin{proof}
The result is trivial for $\nn = 2$. We assume $\nn \ge 3$, and establish the existence of~$\ww_1$ and~$\ww_2$ for an expression~$\ww$ of~$\Dt\nn$ using induction on the combinatorial distance~$\dd$ between~$\Dtt\nn$ and~$\ww$, \ie, the minimal number of relations of~$\PH$ needed to transform~$\Dtt\nn$ to~$\ww$. For~$\dd = 0$, the result is trivial with $\ww_1:=\nobreak \Dtt{\nn - 0.5}$ and~$\ww_2:= \ew$. Assume now~$\dd \ge 1$, and let~$\ww'$ satisfying $\dist(\Dtt\nn, \ww') = \pp - 1$ and $\dist(\ww', \ww) = 1$. Write $\ww' = \uu_1 \vv' \uu_2$, $\ww = \uu_1 \vv \uu_2$ with $\vv' {=} \vv$ a relation of~$\PH$. By induction hypothesis, there exist~$\ww_1', \ww_2'$ satisfying $\ww' = \ww_1' \gh{\kk'} \ww_2'$ with $\ww_1'\ww_2' \eqp \Dtt{\nn - 0.5}$ \linebreak and $\kk':= \nn + \lgg{\ww'_1} - 1$. 
 
We consider the various possibilities for the position of~$\vv'$ inside~$\ww'$. If $\uu_1\vv'$ is a prefix of~$\ww_1'$, \ie, if we have $\ww_1' = \uu_1\vv'\uu_1''$ for some~$\uu''_1$, we find $\ww' = \uu_1\vv'\uu_1'' \gh{\kk'} \ww_2'$ and $\ww = \uu_1\vv \uu_1'' \gh{\kk'} \ww_2'$. Now we have $\uu_1\vv\uu_1''\ww_2' \eqp \uu_1\vv'\uu_1''\ww_2' = \ww_1'\ww_2' \eqp \Dtt{\nn - 0.5}$, whence the result with $\ww_1:= \uu_1\vv\uu_1''$, $\ww_2:= \ww_2'$, and $\kk:= \kk'$. The argument is similar if $\vv'\uu_2$ is a suffix of~$\ww_2'$. 

There remain the cases when $\gh{\kk'}$ occurs in~$\vv'$, \ie, $\gh{\kk'}$ is involved in going from~$\ww'$ to~$\ww$. Assume first that $\vv' {=}\vv$ is a length~$2$ relation. Then $\vv'$ is either $\gh\ii \gh\jj$ with $\jj \ge \ii + 3$, or $\gh\ii \gh\jj$ with $\jj \le \ii - 2$. As $\gh{\kk'}$ occurs in~$\vv'$ in position either~$1$ or~$2$, four cases are to be considered. 
 
\ITEM1 $\gh\ii$ is the last letter of~$\ww_1'$ and $\jj = \kk' \ge \ii + 3$ holds. Putting $\ww_1' :=\ww_1''\gh\ii$, we obtain $\ww' = \ww_1''\gh\ii \gh{\kk'} \ww_2'$ and $\ww = \ww_1'' \gh{\kk' - 1} \gh\ii\ww_2'= \ww_1'' \gh{\nn+\lgg{\ww_1''}-1} \gh\ii \ww_2'$. Moreover, $\ww_1''\gh\ii\ww_2' \eqp \Dtt{\nn - 0.5}$ holds, whence the result for $\ww_1:=\ww_1''$, $\ww_2:=\gh\ii\ww_2'$, and $\kk:= \kk' - 1$. 
 
\ITEM2 $\gh\ii$ is the last letter of~$\ww_1'$ and $\jj = \kk' \le \ii - 2$ holds. This is impossible, because a letter~$\gh\ii$ in this position would contribute $\ii + 1 +\lgg{\ww_2'}$, hence at least $\kk' + 3 + \lgg{\ww_2'}$, \ie, $3\nn - 2$, to the ceiling~$\plf{\ww'}$, which is that of~$\Dt\nn$, namely~$3\nn - 5$. 
 
\ITEM3 $\gh\ii$ is $\gh{\kk'}$ and $\gh\jj$ is the first letter of~$\ww_2'$, with $\jj \ge \ii + 3$. This is impossible for the same ceiling reason as in~\ITEM2.
 
\ITEM4 $\gh\ii$ is $\gh{\kk'}$ and $\gh\jj$ is the first letter of~$\ww_2'$, with $\jj \le \ii - 2$. This is similar to~\ITEM1.
 
\noindent We now similarly handle the case when $\vv' {=}\vv$ is a length~$3$ relation. Then $\vv'$ is either $\gh\ii \gh{\ii + 1} \gh{\ii + 3}$, or $\gh{\ii + 1} \gh\ii \gh{\ii + 1}$. This time, $\gh{\kk'}$ occurs in~$\vv'$ in position~$1$, $2$, or~Ê$3$, so six cases are a priori possible. 
 
\ITEM1 $\gh\ii \gh{\ii + 1}$ is the final factor of~$\ww_1'$, and $\ii + 3 = \kk'$ holds. Putting $\ww_1' :=\nobreak \ww_1''\gh{\kk' - 3}\gh{\kk' - 2}$, we obtain $\ww' = \ww_1'' \gh{\kk' - 3} \gh{\kk' - 2} \gh{\kk'} \ww_2'$ and $\ww = \ww_1''\gh{\kk' - 2} \gh{\kk' - 3} \gh{\kk' - 2} \ww_2'$. Moreover, we find $\ww_1''\gh{\kk' - 3} \gh{\kk' - 2} \ww_2' \eqp \Dtt{\nn - 0.5}$, whence the result for $\ww_1:=\ww_1''$, $\ww_2:=\gh{\kk - 3} \gh{\kk - 2} \ww_2'$, and $\kk:= \kk'$.
 
\ITEM2, \ITEM3 $\gh\ii$ is the last letter of~$\ww_1'$ with $\ii + 1 = \kk'$, and $\gh{\ii + 3}$ is the first letter in~$\ww_2'$, or we have $\ii = \kk'$ and $\gh{\ii + 1} \gh{\ii + 3}$ is a prefix of~$\ww'_2$. These cases are impossible because $\plf{\ww'} = 3\nn-5$ holds, whereas the letter~$\gh{\kk' + 2}$ would contribute $\nn + \lgg{\ww_1'} + 1 + \lgg{\ww_2'} - 1 = \nn +2\nn - 3 - 1 = 3\nn-4$ to the ceiling of~$\ww'$. 
 
\ITEM4, \ITEM5 $\gh{\ii + 1} \gh\ii$ is a suffix of~$\ww_1'$ and $\ii + 1 = \kk$ holds, or $\gh{\ii + 1}$ is the last letter of~$\ww_1'$ and $\ii = \kk$ holds and $\gh{\ii + 1}$ is the first letter of~$\ww_2'$. These cases are impossible because $\gh{\kk+2}$ would contribute $ 3\nn - 3$ to the ceiling of~$\ww'$. 
 
\ITEM6 $\gh{\ii}\gh{\ii + 1}$ is a prefix of~$\ww_2'$, with $\ii + 1 = \kk'$. Putting $\ww_2' = \gh{\kk' - 1} \gh{\kk'} \ww_2''$, we obtain $\ww' = \ww_1' \gh{\kk'} \gh{\kk' - 1} \gh{\kk'} \ww_2''$ and $\ww = \ww_1'\gh{\kk' - 1} \gh{\kk'} \gh{\kk' + 2} \ww_2''$. Moreover, we find $\ww_1' \gh{\kk' - 1} \gh{\kk'} \ww_2'' \eqp \Dtt{\nn - 0.5}$, whence the result for $\ww_1:=\ww_1'\gh{\kk' - 1} \gh{\kk'}$, $\ww_2:=\ww_2''$, and $\kk:= \kk' + 2$. 

This completes the induction. For the final equivalence, we find, using~\eqref{E:Delta1},
$$\ww_1 \gh\kk \ww_2 \eqp \Dtt\nn \eqp \gh{\nn-1}\Dtt{\nn - 0.5} \eqp \gh{\nn-1}\ww_1\ww_2.$$
By right cancelling~$\ww_2$, we deduce $\ww_1 \gh\kk \eqp \gh{\nn-1}\ww_1$.
\end{proof}

We now state a similar result for the expressions of~$\Dt{\nn - 0.5}$. The latter is equal to $\Dt{\nn- 1.5} \gh{3\nn-8} \gh{3\nn-7}$ and the two letters~$\gh{3\nn-8}$ and $\gh{3\nn-7}$ can be moved left. 

\begin{lemm}\label{L:ExpDnhalf}
For $\nn\ge 3$, every expression of~$\Dt{\nn - 0.5}$ has the form $\ww_1\gh\kk \ww_2 \gh\ell \ww_3$ with $\ww_1\ww_2\ww_3 \eqp \Dtt{\nn - 1.5}$, $\kk = \nn-2+\lgg{\ww_1}$, and $\ell = \nn-1+\lgg{\ww_1\ww_2}$; in this case, $\ww_1 \gh\kk \ww_2\ww_3$ represents~$\Dt{\nn-1}$.
\end{lemm}

We skip the proof, which is entirely similar to that of Lemma~\ref{L:ExpDn}---but with more cases, as one has to take care of the positions of two letters.

Applying Lemmas~\ref{L:ExpDn} and~\ref{L:ExpDnhalf}, we can now easily establish various properties of simple elements, paving the way for a partition of these elements in several families.

\begin{lemm}\label{L:Divnn2}
\ITEM1 For $\nn \ge 2$, every left divisor of~$\Dt{\nn}$ either left divides~$\Dt{\nn - 0.5}$, or has the form $\aa \gh\kk \bb$ with $\aa\bb\dive\Dt{\nn - 0.5}$, $\kk = \nn + \lgg\aa - 1$, and $\aa \gh\kk =\gh{\nn-1}\aa$.

\ITEM2 For $\nn\ge 2$ and $\aa\dive\Dt{\nn}$, the conditions $\aa\dive\Dt{\nn - 0.5}$ and $\gh{\nn-1} \not\dive\nobreak \aa$ are equivalent.

\ITEM3 For $\nn \ge 3$, every left divisor of~$\Dt{\nn - 0.5}$ either left divides~$\Dt{\nn-1}$, or has the form $\gh{\nn-2}\gh{\nn-1}\aa$ with $\aa\dive\Dtt{\nn - 1.5}$.
\end{lemm}

\begin{proof}
\ITEM1 If $\cl{\ww} \dive \Dt\nn$ holds, then $\Dt\nn$ has an expression of form~$\ww\ww'$ for some~$\ww'$. By Lemma~\ref{L:ExpDn}, we can write $\ww\ww' = \ww_1 \gh\kk \ww_2$ with $\ww_1\ww_2 \eqp \Dtt{\nn - 0.5}$. Then either $\ww$ is a prefix of~$\ww_1$, and then we have $\cl{\ww} \dive \Dt{\nn - 0.5}$, or $\ww$ has the form $\ww_1 \gh\kk \ww_2'$ with $\ww_2'$ a prefix of~$\ww_2$, and then we have $\cl{\ww} = \cl{\ww_1} \gh\kk \cl{\ww_2'}$ with $\cl{\ww_1}\cl{\ww_2'} = \cl{\ww_1\ww_2'} \dive \Dt{\nn - 0.5}$. Moreover, $\ww_1 \gh\kk \eqp \gh{\nn-1}\ww_1$ implies $\cl{\ww_1} \gh\kk = \gh{\nn-1}\cl{\ww_1}$.

\ITEM2 Assume $\aa \dive \Dt{\nn - 0.5}$. By Lemma~\ref{L:AtomDiv}, $\gh{\nn-1}$ does not left divide~$\Dt{\nn - 0.5}$, so, a fortiori, $\gh{\nn-1}$ does not left divide~$\aa$. 

Conversely, assume $\aa \dive \Dt\nn$ and $\gh{\nn-1} \not\dive \aa$. By~\ITEM1, there are two possibilities: either we have $\aa \dive \Dt{\nn - 0.5}$, as expected, or $\aa$ can be decomposed as $\bb \gh\kk \cc$ with $\bb\cc \dive\Dt{\nn - 0.5}$ and $\bb \gh\kk = \gh{\nn-1}\bb$, implying $\gh{\nn-1} \dive \aa$ and contradicting the assumption. So $\aa\dive\Dt{\nn - 0.5}$ is the only possibility. 

\ITEM3 If $\cl\ww \dive \Dt{\nn - 0.5}$ holds, then $\Dt{\nn - 0.5}$ has an expression of form~$\ww\ww'$ for some~$\ww'$. By Lemma~\ref{L:ExpDnhalf}, we can write $\ww \ww' = \ww_1 \gh\kk \ww_2 \gh\ell \ww_3$ with $\ww_1\ww_2\ww_3\equiv \Dtt{\nn - 1.5}$, \linebreak $\kk =\nobreak \nn -\nobreak 2 +\nobreak \lgg{\ww_1}$, and $\ell = \nn -1 +\lgg{\ww_1\ww_2}$. Then three cases may arise. Either $\ww$ is a prefix of~$\ww_1$, and then we have $\cl{\ww} \dive \Dt{\nn - 1.5}$, whence a fortiori $\cl{\ww} \dive \Dt{\nn-1}$. Or $\ww$ is $\ww_1 \gh\kk \ww_2'$ for some prefix~$\ww_2'$ of~$\ww_2$. By Lemma~\ref{L:ExpDnhalf}, we have $\cl\ww \dive\Dt{\nn-1}$ again. Or $\ww$ is $\ww_1 \gh\kk \ww_2 \gh\ell \ww_3'$ for some prefix~$\ww_3'$ of~$\ww_3$, say $\ww_3=\ww_3'\ww_3''$. Applying~\eqref{E:FNDthalf}, we find $\ww \ww_3'' \eqp\Dtt{\nn - 0.5} \eqp \gh{\nn-2}\gh{\nn-1}\Dtt{\nn- 1.5} \eqp \gh{\nn-2}\gh{\nn-1} \ww_1\ww_2\ww_3'\ww_3''.$
Right cancelling~$\ww''_3$, we deduce $\ww \eqp\gh{\nn-2}\gh{\nn-1} \ww_1\ww_2\ww_3'$, with~$\ww_1, \ww_2, \ww'_3$ satisfying~$\cl{\ww_1\ww_2\ww_3'}\dive\Dt{\nn - 1.5}$.
\end{proof}

\subsection{Partitioning the sets~$\SPP\nn\ell$}\label{SS:Part}

With the preparatory results of Section~\ref{SS:Prepa}, it is now easy to describe the simple elements of the monoid~$\HHp$ more precisely. To this end, we introduce subfamilies of~$\HHp$. We shall eventually see that these subfamilies form a partition of the set~$\SPP\nn\ell$ of all length~$\ell$ left divisors of~$\Dt\nn$.

\begin{defi}
For $\nn\ge 2$ and $0 \le \ell \le 2\nn - 3$, we put
\begin{align*}
&(\tI)\quad \AAAA\nn\ell := \{\aa \mid \aa \dive \Dt{\nn - 1} \ \text{and}\ \lgg\aa = \ell\} \text{\ for $\nn \ge 2, \ell \ge 0$},\\
&(\tII)\quad \BBBB\nn\ell := \{ \gh{\nn-1}\aa \mid \aa \dive \Dt{\nn - 1}\ \text{and}\ \lgg\aa = \ell - 1 \} \text{\ for $\nn \ge 2, \ell \ge 1$},\\
&(\tIII)\quad \CCCC\nn\ell := \{\gh{\nn-2}\gh{\nn-1}\aa \mid \aa \dive \Dt{\nn - 1.5}\ \text{and}\ \lgg\aa = \ell - 2\} \text{\ for $\nn \ge 3, \ell \ge 2$},\\
&(\tIV)\quad \DDDD\nn\ell := \{\gh{\nn-1} \gh{\nn-2} \gh{\nn-1} \aa \mid \gh{\nn-2} \aa \dive \Dt{\nn - 1}\ \text{and}\ \lgg\aa = \ell - 3\} \text{\ for $\nn, \ell \ge 3$},
\end{align*}
completed with $\AAAA\nn\ell = \BBBB\nn\ell = \CCCC\nn\ell = \DDDD\nn\ell := \emptyset$ for other values of~$\nn$ and~$\ell$.
\end{defi}

The first step is to check that the above sets consist of left divisors of~$\Dt\nn$.

\begin{lemm}\label{L:InclusionDelta\nn}
For all $\nn, \ell$, the sets $\AAAA\nn\ell,\BBBB\nn\ell,\CCCC\nn\ell$ et~$\DDDD\nn\ell$ are included in~$\SPP\nn\ell$.
\end{lemm}

\begin{proof}
By definition, all elements of $\AAAA\nn\ell,\BBBB\nn\ell,\CCCC\nn\ell$, and~$\DDDD\nn\ell$ have length~$\ell$, so the point is to check that they left divide~$\Dt\nn$. As $\Dt{\nn - 1}$ left divides~$\Dt\nn$, the result is obvious for~$\AAAA\nn\ell$. Next, $\aa \dive \Dt{\nn - 1}$ implies $\gh{\nn - 1} \aa \dive \gh{\nn - 1} \Dt{\nn - 1}$, whence $\gh{\nn - 1} \aa \dive \gh{\nn - 1} \Dt{\nn - 1} \gh{3\nn - 7} = \Dt\nn$. So $\BBBB\nn\ell$ is included in~$\SP\nn$. Then, by~\eqref{E:FNDthalf}, we have $\Dt\nn = \gh{\nn-2}\gh{\nn-1}\Dt{\nn - 1.5}\gh{3\nn-5}$, so $\aa \dive \Dt{\nn - 1.5}$ implies $\gh{\nn - 2}\gh{\nn - 1}\aa \dive \gh{\nn-2}\gh{\nn-1}\Dt{\nn - 1.5}$, whence $\gh{\nn - 2}\gh{\nn - 1}\aa \dive \Dt\nn$. So $\CCCC\nn\ell$ is included in~$\SP\nn$. Finally, for $\bb = \gh{\nn-1} \gh{\nn-2} \gh{\nn-1} \aa$ with $\gh{\nn - 2}\aa \dive\nobreak \Dt{\nn - 1}$, we have $\Dt{\nn-1} =\gh{\nn-2} \Dt{\nn-1.5}$ by~\eqref{E:Delta1}, whence $\aa \dive\nobreak\Dt{\nn-1.5}$ by left cancelling~$\gh{\nn - 2}$. A direct computation gives $\Dt\nn = \gh{\nn-1} \gh{\nn-2} \gh{\nn-1} \Dt{\nn - 1.5}$, so $\aa \dive \Dt{\nn - 1.5}$ implies $\bb \dive \gh{\nn-1} \gh{\nn-2} \gh{\nn-1} \Dt{\nn - 1.5} = \Dt\nn$. So $\DDDD\nn\ell$ is included in~$\SP\nn$. 
\end{proof}

The second step consists in showing that the various sets $\AAAA\nn\ell \wdots \DDDD\nn\ell$ are pairwise disjoint. This is more delicate, in that it involves proving that certain words are not equivalent. According to Lemma~\ref{L:Rev}\ITEM1, this can be seen using $\PH$-grids.

\begin{lemm}\label{L:EnsDisjoints2à2}
For all $\nn, \ell$, the sets $\AAAA\nn\ell, \BBBB\nn\ell, \CCCC\nn\ell$, and~$ \DDDD\nn\ell$ are pairwise disjoint.
\end{lemm}

\begin{proof}
To prove that $\AAAA\nn\ell$ is disjoint from $\BBBB\nn\ell,\CCCC\nn\ell$, and~$\DDDD\nn\ell$, it suffices to prove that no element of the latter three sets left divides~$\Dt{\nn - 1}$. Now, by definition, $\gh{\nn - 1}$ left divides every element of~$\BBBB\nn\ell$ and~$\DDDD\nn\ell$, whereas, by Lemma~\ref{L:AtomDiv}, $\gh{\nn - 1}$ does not left divide~$\Dt{\nn - 1}$. So $\BBBB\nn\ell$ and~$\DDDD\nn\ell$ are disjoint from~$\AAAA\nn\ell$.

Next, a direct computation gives $(\gh{\nn - 2} \gh{\nn - 1}, \Dtt{\nn - 1}) \rev (\gh{3\nn - 7}, \Dtt{\nn - 1.5})$, which, by Lemma~\ref{L:Rev}\ITEM1, proves $\gh{\nn - 2} \gh{\nn - 1} \not\dive \Dt{\nn - 1}$. As $\gh{\nn - 2} \gh{\nn - 1}$ left divides every element of~$\CCCC\nn\ell$, it follows that $\CCCC\nn\ell$ is disjoint from~$\AAAA\nn\ell$.

Assume now $\aa \in \BBBB\nn\ell\cap\CCCC\nn\ell$. Then, by definition, we have both $\gh{\nn - 1} \dive \aa$ and $\aa \dive \gh{\nn-2} \gh{\nn-1} \Dt{\nn - 1.5}$, whence $\gh{\nn-1} \dive \gh{\nn-2} \gh{\nn-1} \Dt{\nn - 1.5}$. This is impossible: a direct computation gives $(\gh{\nn - 1}, \gh{\nn - 2} \gh{\nn - 1} \Dtt{\nn - 1.5}) \rev (\gh{3\nn - 10}, \gh{\nn - 1} \gh{\nn - 2} \Dtt{\nn - 1.5})$, which, by Lemma~\ref{L:Rev}\ITEM2, proves $\gh{\nn-1} \not\dive \gh{\nn-2} \gh{\nn-1} \Dt{\nn - 1.5}$. Hence $\BBBB\nn\ell$ and $\CCCC\nn\ell$ are disjoint.

Assume next $\aa \in \BBBB\nn\ell\cap\DDDD\nn\ell$. We have both $\aa = \gh{\nn-1}\bb$ with~$\bb \in \SPP{\nn - 1}{\ell - 1}$, and $\aa =\gh{\nn-1} \gh{\nn-2} \gh{\nn-1} \cc$ with $\gh{\nn-2}\cc \dive \Dt{\nn - 1}$. By left cancelling~$\gh{\nn - 1}$, we deduce $\bb = \gh{\nn-2} \gh{\nn-1} \cc$, whence $\gh{\nn-2} \gh{\nn-1} \cc \dive \Dt{\nn - 1}$ and, a fortiori, $\gh{\nn-2} \gh{\nn-1} \dive \Dt{\nn - 1}$, what we saw above is false. Hence $\BBBB\nn\ell$ and $\DDDD\nn\ell$ are disjoint.

Finally, assume $\aa \in \CCCC\nn\ell\cap\DDDD\nn\ell$. By definition, we have $\aa = \gh{\nn-2}\gh{\nn-1} \bb =\gh{\nn-1} \gh{\nn-2} \gh{\nn-1} \cc$ with~$\bb \dive \Dt{\nn - 1.5}$ and $\gh{\nn-2}\cc \dive \Dt{\nn - 1}$. As $\gh{\nn-1} \gh{\nn-2} \gh{\nn-1}$ is also $\gh{\nn-2} \gh{\nn-1} \gh{\nn+1}$, we deduce $\gh{\nn-2} \gh{\nn-1} \bb = \gh{\nn-2} \gh{\nn-1} \gh{\nn+1} \cc$, whence $\bb = \gh{\nn+1} \cc$ by left cancelling $\gh{\nn-2} \gh{\nn-1}$, and, from there, $\gh{\nn+1} \dive \bb \dive \Dt{\nn - 1.5}$. Now, by Lemma~\ref{L:AtomDiv}, $\gh{\nn+1}$ does not left divide~$\Dt{\nn - 1.5}$. Hence $\CCCC\nn\ell$ and $\DDDD\nn\ell$ are disjoint.
\end{proof}

We are now ready to establish the expected partition result:

\begin{prop}\label{P:Part}
For all $\nn, \ell$, the sets~$\AAAA\nn\ell$, $\BBBB\nn\ell$, $\CCCC\nn\ell$, and~$\DDDD\nn\ell$ form a partition of~$\SPP\nn\ell$.
\end{prop}

\begin{proof}
Owing to Lemmas~\ref{L:InclusionDelta\nn} and~\ref{L:EnsDisjoints2à2}, the only point remaining to be proved is that every element of~$\SPP\nn\ell$ belongs to one of the sets~$\AAAA\nn\ell$, $\BBBB\nn\ell$, $\CCCC\nn\ell$, $\DDDD\nn\ell$. So let~$\aa$ belong to~$\SPP\nn\ell$. By Lemma~\ref{L:Divnn2}\ITEM1, we have either $\aa \dive \Dt{\nn - 0.5}$, or $\aa = \bb \gh\kk \cc$ with $\bb\cc\dive\Dt{\nn - 0.5}$ and $\bb \gh\kk = \gh{\nn - 1}\bb$. Assume first $\aa \dive \Dt{\nn - 0.5}$. By Lemma~\ref{L:Divnn2}\ITEM3, we have either $\aa \dive \Dt{\nn-1}$, whence $\aa\in\AAAA\nn\ell$, or $\aa = \gh{\nn-2} \gh{\nn-1} \dd$ with $\dd \dive \Dt{\nn - 1.5}$, whence~$\aa\in\CCCC\nn\ell$.

Assume now $\aa = \bb \gh\kk \cc$ with $\bb\cc \dive\Dt{\nn - 0.5}$ and $\bb \gh\kk = \gh{\nn - 1}\bb$, whence $\aa = \gh{\nn-1} \bb\cc$. By Lemma~\ref{L:Divnn2}\ITEM3, we have either $\bb\cc \dive \Dt{\nn-1}$, whence $\aa\in\BBBB\nn\ell$, or $\bb\cc = \gh{\nn-2}\gh{\nn-1}\dd$ with~$\dd \dive \Dt{\nn - 1.5}$. In the latter case, we find $\aa = \gh{\nn-1} \gh{\nn-2} \gh{\nn-1} \dd$. Moreover, $\dd\dive\nobreak\Dt{\nn - 1.5}$ implies $\gh{\nn-2}\dd \dive \gh{\nn-2} \Dt{\nn - 1.5} = \Dt{\nn-1}$, whence $\aa\in\DDDD\nn\ell$. \end{proof}

With the partition of Proposition~\ref{P:Part}, we can now count the left divisors of~$\Dt\nn$. 

\begin{lemm}\label{L:Bij}
For $\nn \ge 3$, let $\BijA\nn\ell$ be the identity map on~$\SPP{\nn - 1}{\ell}$, let $\BijB\nn\ell$ be the map $\aa \mapsto \gh{\nn - 1} \aa$ on~$\SPP{\nn - 1}{\ell - 1}$, and let $\BijC\nn\ell$ be the map on~$\SPP{\nn - 1}{\ell - 2}$ defined by $\FF(\aa):= \gh{\nn - 2} \gh{\nn - 1} \aa$ if $\aa \dive \Dt{\nn - 1.5}$ holds, and $\FF(\aa):= \gh{\nn - 1} \gh{\nn - 2} \gh{\nn - 1} \bb$ with $\aa = \gh{\nn - 2}\bb$ otherwise. Then $\BijA\nn\ell$, $\BijB\nn\ell$, and~$\BijC\nn\ell$ respectively establish bijections
$$\SPP{\nn - 1}{\ell} \leftrightarrow \AAAA\nn\ell, \quad\SPP{\nn - 1}{\ell - 1} \leftrightarrow \BBBB\nn\ell, \quad\text{and}\quad\SPP{\nn - 1}{\ell - 2} \leftrightarrow \CCCC\nn\ell \cup \DDDD\nn\ell.$$
\end{lemm}

\begin{proof}
The result for~$\BijA\nn\ell$ directly follows from the definition of~$\AAAA\nn\ell$. For~$\BijB\nn\ell$, it follows from the definition of~$\BBBB\nn\ell$ and the left cancellativity of~$\HHp$, which ensures that $\BijB\nn\ell$ is injective. Finally, for~$\BijC\nn\ell$, put
$$\XX_1:= \{\aa \in \SPP{\nn - 1}{\ell - 2} \mid \aa \dive \Dt{\nn - 1.5}\} \quad\text{and}\quad \XX_2:= \{\aa \in \SPP{\nn - 1}{\ell - 2} \mid \aa \not\dive \Dt{\nn - 1.5}\}.$$
It follows from the definition of~$\CCCC\nn\ell$ and the left cancellativity of~$\HHp$ that $\BijC\nn\ell$ establishes a bijection from~$\XX_1$ to~$\CCCC\nn\ell$. On the other hand, for~$\aa$ in~$\XX_1$, Lemma~\ref{L:Divnn2}\ITEM2 implies~$\gh{\nn - 2} \dive \aa$, say $\aa = \gh{\nn - 2}\bb$, and then the left cancellativity of~$\HHp$ implies that $\BijC\nn\ell$ establishes a bijection from~$\XX_2$ to~$\DDDD\nn\ell$. As $\CCCC\nn\ell$ and $\DDDD\nn\ell$ are disjoint, this completes the proof.
\end{proof}

Lemma~\ref{L:Bij} immediately implies that, if we denote by~$\NPP\nn\ell$ the cardinal of~$\SPP\nn\ell$, then the numbers~$\NPP\nn\ell$ are determined by the inductive rule
\begin{equation}\label{E:DenRec}
\NPP\nn\ell =\NPP{\nn-1}{\ell}+\NPP{\nn-1}{\ell-1}+\NPP{\nn-1}{\ell-2},
\end{equation}
starting from the initial values $\NPP20 = \NPP21 = 1$. It follows that the numbers~$\NPP\nn\ell$ appear in the generalized Pascal triangle in which each entry is the sum of the three entries above it, starting from the row~$(1, 1)$, see Figure~\ref{F:Triangle}. An obvious induction from~\eqref{E:DenRec} shows that $\NPP\nn\ell$ is the coefficient of~$\xx^{\ell - 1}$ in~$(1 + \xx)(1 + \xx + \xx^2)^{\nn - 2}$, that $\NPP\nn\ell = \NPP\nn{2\nn - 3 - \ell}$ holds for~$\nn - 1 \le \ell \le 2\nn - 3$, and that, for $0 \le \ell \le \nn - 2$, the number~$\NPP\nn\ell$ is the number of (compact-rooted) directed animals of size~$\nn - 1$ with $\nn - 1 - \ell$ source points, see~\cite[Table~1]{GoV} and~\cite[sequence 005773]{OEIS}. In particular, the highest value occurring in the $\nn-1$st row of Figure~\ref{F:Triangle} (the one that corresponds to the divisors of~$\Dt\nn$), namely~$\NPP\nn{\nn - 2}$ and~$\NPP\nn{\nn - 1}$,---that is, the sequence $1$, $2$, $5$, $13$, $35$, ...---is the number of directed animals of size~$\nn - 1$ with one source point. Finding an explicit direct bijection between the divisors of~$\Dt\nn$ in~$\HHp$ and size~$\nn - 1$ directed animals~\cite{GoV, Vie}---or, equivalently, ``arbres guingois'' or bicolored Motzkin paths~\cite{BeP}---is a natural open question. 

From~\eqref{E:DenRec} again, it is clear that the total number of left divisors of~$\Dt\nn$ triples when one goes from a row of the triangle to the next one, and, as $\Dt2$ admits two left divisors, we obtain

\begin{prop}\label{P:Denombr}
For $\nn\ge 2$, the number of left divisors of~$\Dt\nn$ in~$\HHp$ is~$2 \cdot 3^{\nn-2}$.
\end{prop}

The number of simple elements of index~$\nn$ is $\sum_\ell\NPP\nn\ell - \sum_\ell\NPP{\nn - 1}\ell$, hence it is $4 \cdot 3^{\nn-3}$: so $2/3$ of the left divisors of~$\Dt\nn$ have index~$\nn$, whereas $1/3$ has index~${<}\,\nn$.

\begin{figure}[htb]
\begin{picture}(72,30)(0,0)
\setlength{\unitlength}{0.8mm}
\psset{unit=0.8mm}
\put(40,32){$1$}
\put(50,32){$1$}
\put(30,24){$1$}
\put(40,24){$2$}
\put(50,24){$2$}
\put(60,24){$1$}
\put(20,16){$1$}
\put(30,16){$3$}
\put(40,16){$5$}
\put(50,16){$5$}
\put(60,16){$3$}
\put(70,16){$1$}
\put(10,8){$1$}
\put(20,8){$4$}
\put(30,8){$9$}
\put(40,8){$13$}
\put(50,8){$13$}
\put(60,8){$9$}
\put(70,8){$4$}
\put(80,8){$1$}
\put(0,0){$1$}
\put(10,0){$5$}
\put(20,0){$14$}
\put(30,0){$26$}
\put(40,0){$35$}
\put(50,0){$35$}
\put(60,0){$26$}
\put(70,0){$14$}
\put(80,0){$5$}
\put(90,0){$1$}
\pspolygon[linearc=2,linewidth=0.5pt](31,5)(15,19.8)(47,19.8)(31,5)
\end{picture}
\caption{\sf Generalized Pascal triangle generating the numbers~$\NPP\nn\ell$: each entry is the sum of the three entries above it: for instance, we find $\NPP52 = 9 = 1 + 3 + 5 = \NPP40 + \NPP41 + \NPP42$; missing values are~$0$.}
\label{F:Triangle}
\end{figure}
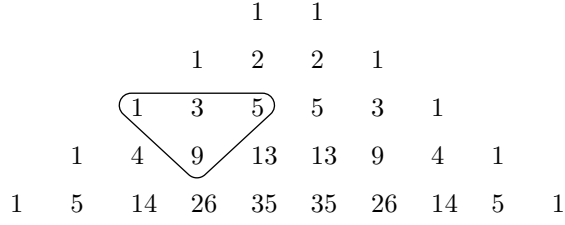

\section{Normal form of simple elements}\label{S:NF}

We complete the investigation of simple elements in~$\HHp$ by determining their normal form. In Section~\ref{SS:GarF}, we saw that normal elements of~$\FFp$ are those, whose normal form is a word in which the indices of the generators decrease, which amounts to saying that a word is the normal form of a simple element if, and only if, it has no factor~$\gh\ii \gh\jj$ with $\jj \ge \ii$. We shall establish below a similar result characterizing the normal form of simple elements in terms of forbidden factors of length~$2$ and~$3$.

\subsection{The key lemma}\label{SS:Cle}

A direct inspection shows that the normal forms of the six simple elements of index~${\le}\,3$, \ie, of the six left divisors of~$\Dt3$, are $\ew$, $\gh1$, $\gh2$, $\gh1\gh2$, $\gh2\gh1$, and~$\gh2\gh1\gh2$. The following result will then enable one to inductively determine the normal form of a simple element according to its position in the partition of Proposition~\ref{P:Part}. 

\begin{lemm}\label{L:Key}
For every simple element~$\aa$ of index $\nn \ge 4$ in~$\HHp$, there exists~$\bb$ of index ${<}\, \nn$ such that exactly one of the following holds:
\begin{align}
\label{E:Main2}
&\hbox to 20mm{\HS{-1}$(\tII)$\hfil} \bb \dive \Dt{\nn - 1} &&\HS{-3}\text{and}\quad \NF(\aa) = \gh{\nn-1}\NF(\bb),\\
\label{E:Main3}
&\hbox to 20mm{\HS{-1}$(\tIII)$\hfil} \bb \dive \Dt{\nn - 1.5} &&\HS{-3}\text{and}\quad \NF(\aa) = \gh{\nn - 2}\gh{\nn-1}\NF(\bb),\\
\label{E:Main4}
&\hbox to 20mm{\HS{-1}$(\tIV)$\hfil} \gh{\nn - 2}\bb \dive \Dt{\nn - 1} &&\HS{-3}\text{and}\quad \NF(\aa)= \gh{\nn-1}\gh{\nn - 2}\gh{\nn-1}\NF(\bb).
\end{align}
\end{lemm}

\begin{proof}
Let $\ell:= \lgg\aa$. By assumption, $\aa$ belongs to~$\SPP\nn\ell \setminus \SPP{\nn - 1}\ell$. Then, by Proposition~\ref{P:Part}, $\aa$ belongs to exactly one of $\BBBB\nn\ell$, $\CCCC\nn\ell$, or~$\DDDD\nn\ell$. So there exists~$\bb$ such that exactly one of the following holds:
\begin{align}
\label{E:Main5}
&\hbox to 20mm{\HS{-3}$(\tII)$\hfil} \bb \dive \Dt{\nn - 1} &&\HS{-3}\text{and}\quad \aa = \gh{\nn-1}\,\bb,\\
\label{E:Main6}
&\hbox to 20mm{\HS{-3}$(\tIII)$\hfil} \bb \dive \Dt{\nn - 1.5} &&\HS{-3}\text{and}\quad \aa = \gh{\nn - 2}\gh{\nn-1}\,\bb,\\
\label{E:Main7}
&\hbox to 20mm{\HS{-3}$(\tIV)$\hfil} \gh{\nn - 2}\bb \dive \Dt{\nn - 1} &&\HS{-3}\text{and}\quad \aa= \gh{\nn-1}\gh{\nn - 2}\gh{\nn-1}\,\bb.
\end{align}
In the case of~\eqref{E:Main5}, we have $\bb \dive \Dt{\nn - 1}$, so $\bb$ is simple with $\ind\bb \le \nn - 1$. In the case of~\eqref{E:Main6}, we have $\bb \dive \Dt{\nn - 1.5} \dive \Dt{\nn - 1.5} \gh{3\nn - 8} = \Dt{\nn - 1}$, so, again, $\bb$ is simple with $\ind\bb \le \nn - 1$. Finally, in the case of~\eqref{E:Main7}, $\gh{\nn - 2}\bb \dive \Dt{\nn - 1}$ implies $\bb \dive\Dt{\nn - 1.5}$ by~\eqref{E:Delta1}, whence $\bb \dive \Dt{\nn - 1}$, so $\bb$ is simple with $\ind\bb \le \nn - 1$. So, in every case, $\bb$ is simple with $\ind\bb < \nn = \ind\aa$. Then, by definition of~$\NF$ and by Proposition~\ref{L:AtomDiv}, \begin{equation}\label{E:Main8}
\text{$\NF(\bb)$ is an $\EH$-reduced word and its first letter is among~$\gh1\wdots\gh{\nn-2}$}.
\end{equation} 

In the case of~\eqref{E:Main5}, \eqref{E:Main8} implies that $\gh{\nn - 1}\NF(\bb)$ is $\EH$-reduced, hence it must the normal form of~$\gh{\nn - 1}\bb$, \ie, of~$\aa$, and \eqref{E:Main2} is true. In the case of~\eqref{E:Main6}, \eqref{E:Main8} implies that $\gh{\nn - 2}\gh{\nn - 1}\NF(\bb)$ is $\EH$-reduced, hence it must the normal form of~$\gh{\nn - 2}\gh{\nn - 1}\bb$, \ie, of~$\aa$, and \eqref{E:Main2} is true. Finally, in the case of~\eqref{E:Main7}, \eqref{E:Main8} implies that $\gh{\nn - 1}\gh{\nn - 2}\gh{\nn - 1}\NF(\bb)$ is $\EH$-reduced, hence it must the normal form of~$\gh{\nn - 1} \gh{\nn - 2} \gh{\nn - 1} \bb$, \ie, of~$\aa$, and \eqref{E:Main2} is true. 
\end{proof}

An easy application of Lemma~\ref{L:Key} is that, in addition to the obstructions of~$\LI$, certain factors cannot appear in the normal form of a simple element. 

\begin{lemm}\label{L:Intp}
Put
\begin{equation}\label{E:FNObss}
\LIp:= \{\gh\ii^2 \mid \ii \ge 1\} \cup \{\gh\ii\gh{\ii+2} \mid \ii \ge 1\} \cup \{\gh\ii\gh{\ii+1}\gh\ii \mid \ii \ge 1\} \cup \{\gh\ii\gh{\ii+1}\gh{\ii+2}\mid \ii \ge 1\}.
\end{equation}
Then the normal form of a simple element of~$\HHp$ contains no factor in~$\LIp$.
\end{lemm}	

\begin{proof}
We prove the result for a simple element~$\aa$ using induction on the index~$\nn$ of~$\aa$. For $\nn \le 3$, a direct inspection of the six possible words gives the result. Assume $\nn \ge 4$. By Lemma~\ref{L:Key}, there exists~$\bb$ simple of index~${<}\nn$ such that exactly one of~\eqref{E:Main2}, \eqref{E:Main3}, or~\eqref{E:Main4} holds. By induction hypothesis, the word~$\NF(\bb)$ contains no factor of~$\LIp$, and we only have to check that the letters added to transform~$\NF(\bb)$ into~$\NF(\aa)$ create no factor in~$\LIp$. As the index of~$\bb$ is~$< \nn$, Lemma~\ref{L:AtomDiv} guarantees that the first letter of~$\NF(\bb)$ must be among~$\gh1 \wdots \gh{\nn - 2}$.

In the case of~\eqref{E:Main2}, $\NF(\aa)$ begins with~$\gh{\nn - 1} \gh\jj$ with $1 \le\nobreak \jj \le\nobreak \nn - 2$: this length~$2$ word is not in~$\LIp$, and it is not the prefix of a word of~$\LIp$ either. Similarly, in the case of~\eqref{E:Main3}, $\NF(\aa)$ begins with~$\gh{\nn - 2} \gh{\nn - 1} \gh\jj$ with $1 \le \jj \le \nn - 3$, and this length~$3$ word includes no factor in~$\LIp$, nor can it contribute to a factor in~$\LIp$. Finally, in the case of~\eqref{E:Main4}, $\NF(\aa)$ begins with~$\gh{\nn - 1} \gh{\nn - 2} \gh{\nn - 1} \gh\jj$ with $1 \le \jj \le \nn - 3$, and, again, this length~$4$ word includes no factor in~$\LIp$, nor can it either contribute to a factor in~$\LIp$. So, in every case, the word~$\NF(\aa)$ has no factor in~$\LIp$.
\end{proof}

We use Lemma~\ref{L:Key} once more to establish a constraint about the first letter of a normal word.

\begin{lemm}\label{L:PremDiv}
If $\aa \dive \Dt\nn$ and $\gh{\nn - 1} \dive \aa$ hold, the first letter of~$\NF(\aa)$ is~$\gh{\nn - 1}$.
\end{lemm}

\begin{proof}
Assume $\aa \dive \Dt\nn$ and $\gh{\nn - 1} \dive \aa$. So~$\aa$ is simple with $\ind\aa \le \nn$. If we had $\ind\aa \le \nn - 1$, hence $\aa \dive \Dt{\nn - 1}$, then $\gh{\nn - 1} \dive \aa$ would be impossible. So we must have $\ind\aa = \nn$. For $\nn \le 3$, a direct inspection of the six possible normal words shows that the result is true. Otherwise, we apply Lemma~\ref{L:Key}. In the cases~\eqref{E:Main2} and~\eqref{E:Main4}, $\NF(\aa)$ explicitly begins with~$\gh{\nn - 1}$. There remains the case of~\eqref{E:Main3}. Assume $\aa = \gh{\nn - 2} \gh{\nn - 1} \bb$ with $\bb \dive \Dt{\nn - 1.5}$, let $\ww$ represent~$\bb$. By constructing a $\PH$-grid from~$(\gh{\nn - 1}, \gh{\nn - 2} \gh{\nn - 1} \ww)$, we see that $\gh{\nn - 1} \dive \aa$ is equivalent to $\gh{\nn + 1} \dive \bb$, hence it implies $\gh{\nn + 1} \dive \bb \dive \Dt{\nn - 1.5} \dive \Dt{\nn - 1.5}\gh{3n - 8} = \Dt{\nn - 1}$, which contradicts Lemma~\ref{L:AtomDiv}. So $\gh{\nn - 1} \dive \aa$ excludes~\eqref{E:Main3}. 
\end{proof}

\subsection{The normal form of simple elements}\label{SS:FNSimple}

Our goal is now to establish that the necessary condition of Lemma~\ref{L:Intp} is also sufficient, thus obtaining a combinatorial characterization of the normal form of simple elements. We begin with a preliminary observation about the indices of the generators~$\gh\ii$ that may appear in words with no factor in~$\LIp$.

\begin{defi}
We put $\HT\ew:= 0$, and, for~$\ww$ nonempty in~$\AHs$, we write~$\HT\ww$ for the largest~$\ii$ such that $\gh\ii$ occurs in~$\ww$. 
\end{defi}

\begin{lemm}\label{L:ObsHaut}
If $\gh\ii \vv$ is $\EH$-reduced with no factor in~$\LIp$, then $\HT\vv \le \ii + 1$ holds. 
\end{lemm}

\begin{proof}
We use induction on~$\lgg\vv$. For $\lgg\vv = 0$, the result is vacuously true. Assume $\lgg\vv \ge 1$, and write $\vv = \gh\jj \ww$. As $\gh\ii \vv$, \ie, $\gh\ii\gh\jj \ww$, is $\EH$-reduced, it contains no factor in~$\LI$, hence $\jj \ge \ii + 3$ is excluded. On the other hand, as $\gh\ii \vv$ has no factor in~$\LIp$, the values $\jj = \ii$ and $\jj = \ii + 2$ are impossible. So the only possible values for~$\jj$ are $1\wdots\ii - 1$, and~$\ii+1$.
	
Assume first $\jj \le \ii-1$. As a factor of~$\gh\ii \vv$, the word~$\gh\jj \ww$ is reduced with no factor in~$\LIp$. Then the induction hypothesis implies $\HT{\ww} \le \jj + 1$, whence $\HT\vv = \max(\jj, \HT\ww) \le \jj + 1 \le \ii + 1$, as expected.
	
Assume now $\jj = \ii + 1$. The result is true for $\lgg\vv = 1$: the word~$\gh\ii \gh{\ii + 1}$ has no factor in~$\LIp$ and its height is~$\ii + 1$. Assume now $\lgg\vv \ge 2$, and write $\vv = \gh{\ii+1}\gh\kk\ww$. As $\vv$ has no factor in~$\LI$, the values $\kk \ge \jj + 3 = \ii + 4$ are forbidden, and, as $\gh\ii\vv$ has no factor in~$\LIp$, the values $\kk = \ii$, $\kk = \ii + 1$, and $\kk = \ii + 2$ are also excluded. So we must have $\kk \le \ii-1$. As $\gh\kk\ww$ is reduced with no factor in~$\LIp$, the induction hypothesis implies $\HT{\ww} \le \kk + 1$, whence $\HT{\vv} = \max(\ii + 1, \kk, \HT\ww) \le \max(\ii + 1, \kk + 1) = \ii + 1$, as expected.
\end{proof}

Completing the characterization of the normal forms of simple elements then relies on a long inductive argument. 	
\begin{lemm}\label{L:ObsNorm}
If $\uu$ is a reduced word of~$\AHs$ with no factor in~$\LIp$, then $\uu$ is the normal form of a simple element with index at most~$\HT\uu + 1$.
\end{lemm}	

\begin{proof}
We will show using induction on $\mm \ge 0$ that, if $\uu$ is an $\EH$-reduced word with no factor in~$\LIp$ and satisfying $\HT\uu = \mm$, then $\cl\uu \dive \Dt{\mm + 1}$ holds. This will imply that $\cl\uu$ is simple with index ${\le}\, \HT\uu + 1$, giving the expected result when $\mm$ varies. So, herafter, we assume that $\uu$ is $\EH$-reduced, has no factor in~$\LIp$, and satisfies $\HT\uu = \mm$; our aim is to establish $\cl\uu \dive \Dt{\mm + 1}$. As can be expected, the various types of Proposition~\ref{P:Part} will appear when we consider the possible cases.

For $\mm = 0$, the word~$\uu$ must be empty. We then find $\cl\uu = 1 \dive \Dt1 = \Dt{\mm + 1}$, as expected. For $\mm = 1$, the only letter occurring in~$\uu$ is~$\gh1$, so~$\uu$ is~$\gh1^\ell$ for some~$\ell \ge 1$. The assumption that $\uu$ has no factor in~$\LIp$ requires $\ell = 1$, whence $\uu = \gh1$. We then find $\cl\uu = \gh1 \dive \gh1 = \Dt2 = \Dt{\mm + 1}$, as expected.

From now on, we assume $\mm \ge 2$. The word~$\uu$ cannot be empty, so it has a first letter, say~$\gh\ii$. By assumption, we have $\mm = \HT\uu$, hence $\ii \le \mm$. On the other hand, Lemma~\ref{L:ObsHaut} implies $\HT\uu \le \ii + 1$, hence $\mm \le \ii + 1$. Therefore, $\uu$ must begin either by~$\gh\mm$, or by~$\gh{\mm - 1}$.

\medskip\noindent{\bf Case 1.} The first letter of~$\uu$ is $\gh{\mm-1}$, say $\uu = \gh{\mm - 1} \vv$. The word~$\vv$ cannot be empty, for otherwise we would have $\uu = \gh{\mm - 1}$ and $\HT\uu = \mm - 1$, contradicting the assumption. Let $\gh\jj$ be the first letter of~$\vv$. By definition, we have $\jj \le \HT\uu = \mm$. Moreover, $\uu$ has no factor in~$\LIp$, so $\jj = \mm - 1$ is impossible. On the other hand, $\vv$, as a factor of~$\uu$, is $\EH$-reduced and has no factor in~$\LIp$, so Lemma~\ref{L:ObsHaut} implies $\HT\vv \le \jj + 1$, and $\jj \le \mm - 2$ would imply $\HT\uu \le \max(\mm - 1, \HT\vv) \le \mm - 1$, contradicting the assumption $\mm = \HT\uu$. So the only possibility is $\jj = \mm$, \ie, $\uu$ begins with~$\gh{\mm - 1}\gh\mm$, say $\uu = \gh{\mm-1}\gh\mm\ww$. 

If $\ww$ is the empty word, we have $\uu = \gh{\mm - 1}\gh\mm$. Applying~\eqref{E:Delta1} twice gives 
$\Dt{\mm+1} =\gh{\mm-1} \gh{\mm} \Dt{\mm-0.5} \gh{3\mm-2}$, which implies $\cl\uu\dive\Dt{\mm+1}$, as expected. 

Assume now that $\ww$ is nonempty, and let $\gh\kk$ be its first letter. As $\ww$ is a factor of~$\uu$, we must have $\kk \le \mm$. As $\uu$, and its factor~$\vv$, have no factor in~$\LIp$, the values $\kk = \mm - 1$ and $\kk = \mm$ are impossible as they would respectively create some factor $\gh{\mm - 1}\gh\mm\gh{\mm - 1}$ and $\gh\mm^2$. So, we necessarily have $\kk <\nobreak \mm - 1$. Since $\ww$, as a factor of~$\uu$, is $\EH$-reduced and has no factor in~$\LIp$, Lemma~\ref{L:ObsHaut} implies $\HT\ww \le \kk + 1$, whence $\HT\ww \le \mm - 1$. The word~$\ww$ is $\EH$-reduced with no factor in~$\LIp$, so the induction hypothesis implies $\cl\ww \dive \Dt{\HT\ww + 1}$, hence a fortiori $\cl\ww \dive \Dt\mm$. Moreover, we know that the first letter of~$\ww$ is not~$\gh{\mm -1}$. By Lemma~\ref{L:PremDiv}, we deduce $\gh{\mm - 1} \not\dive \cl\ww$, and then, by Lemma~\ref{L:Divnn2}\ITEM2, $\cl\ww \dive \Dt{\mm - 0.5}$. By definition, this means that $\cl\uu$ belongs to~$\CCCC{\mm + 1}{\lgg\uu}$ and, therefore, implies $\cl\uu \dive \Dt{\mm + 1}$, as expected. 
	
\medskip\noindent{\bf Case 2.} The first letter of~$\uu$ is $\gh{\mm}$, say $\uu = \gh{\mm} \vv$. If $\vv$ is empty, we have $\uu = \gh\mm$, which has height~$\mm$, and $\cl\uu = \gh\mm \dive \Dt{\mm + 1}$, as expected. 

We now suppose $\vv$ nonempty. Let~$\gh\jj$ be its first letter. The assumption $\HT\uu = \mm$ implies $\jj \le \mm$. As $\uu$ has no factor in~$\LIp$, the value $\jj = \mm$ is impossible, since it would create an initial factor~$\gh\mm^2$. So we have $\jj \le \mm - 1$.
	 
\smallskip{\bf Subcase 2.1.} We have $\jj \le \mm - 2$. Then Lemma~\ref{L:ObsHaut} implies $\HT\vv \le \mm - 1$. Moreover, as a factor of~$\uu$, the word~$\vv$ is $\EH$-reduced and has no factor in~$\LIp$. The induction hypothesis then implies $\cl{\vv}\dive\Dt{\HT\vv + 1}$, hence a fortiori $\cl{\vv} \dive \Dt\mm$. Therefore, we have $\cl\uu = \gh\mm \cl{\vv}$ with $\cl{\vv}\dive\Dt\mm$. This means that $\cl\uu$ lies in~$\BBBB{\mm + 1}{\lgg\uu}$, implying $\cl\uu \dive \Dt{\mm + 1}$, as expected. 
	 
\smallskip{\bf Subcase 2.2.} We have $\jj = \mm - 1$. Write $\vv = \gh{\mm - 1}\ww$, yielding $\uu = \gh{\mm}\gh{\mm - 1}\ww$.

If $\ww$ is empty, we have $\uu = \gh{\mm}\gh{\mm - 1}$. Applying~\eqref{E:Delta1} twice gives the equality $\Dt{\mm+1} = \gh{\mm} \gh{\mm-1}\Dt{\mm-1}\gh{3\nn-7} \gh{3\mm-4}$, whence $\cl\uu\dive\Dt{\mm+1}$, as expected. 

We assume now that $\ww$ is nonempty, with first letter~$\gh\kk$. The assumption $\HT\uu =\nobreak \mm$ implies $\kk \le \mm$. Moreover, as $\uu$ has no factor in~$\LIp$, the value $\kk = \mm - 1$ is impossible, since it would create in position~$2$ a factor~$\gh{\mm - 1}^2$.

\smallskip{\bf Subsubcase 2.2.1.} We have $\kk \le \mm - 2$. As a factor of~$\uu$, the word~$\ww$ is $\EH$-reduced and has no factor in~$\LIp$, so Lemma~\ref{L:ObsHaut} implies $\HT\ww \le \mm -1$, whence $\HT\vv = \mm - 1$. As a factor of~$\uu$, the word~$\vv$ is $\EH$-reduced and has no factor in~$\LIp$, so the induction hypothesis implies $\cl{\vv}\dive\Dt{\HT\vv + 1}$, hence a fortiori $\cl{\vv}\dive\Dt\mm$. This means that $\cl\uu$ lies in~$\BBBB{\mm + 1}{\lgg\uu}$ and implies $\cl\uu \dive \Dt{\mm + 1}$, as expected.

\smallskip{\bf Subsubcase 2.2.2.} We have $\kk = \mm$. Write $\ww = \gh\mm \uu'$, yielding $\uu = \gh\mm\gh{\mm-1}\gh\mm\uu'$. If $\uu'$ is empty, we have $\uu = \gh\mm\gh{\mm-1}\gh\mm$. A direct computation from~\eqref{E:Delta1} gives $\Dt{\mm+1} = \gh{\mm} \gh{\mm-1} \gh\mm \Dt{\mm-1} \gh{3\nn-7}$, whence $\cl\uu \dive \Dt{\mm + 1}$, as expected. 

We assume now that $\uu'$ is nonempty, with first letter~$\gh\ell$. The assumption $\HT\uu =\nobreak \mm$ implies $\ell \le \mm$. The assumption that $\uu$ has no factor in~$\LIp$ excludes $\ell = \mm - 1$ and $\ell = \mm$, as these values would create factors $\gh{\mm - 1}\gh\mm \gh{\mm - 1}$ or~$\gh\mm^2$ in~$\uu$. Next, Lemma~\ref{L:ObsHaut} implies $\HT{\uu'} \le \mm - 1$. Moreover, as a factor of~$\uu$, the word~$\uu'$ is $\EH$-reduced and has no factor in~$\LIp$, so the induction hypothesis implies $\cl{\uu'} \dive \Dt\mm$. By Lemma~\ref{L:PremDiv}, if we had $\gh{\mm - 1} \dive \cl{\uu'}$, the first letter of~$\uu'$ should be~$\gh{\mm - 1}$, contradicting $\ell \le \mm - 2$. Hence we have $\gh{\mm - 1} \not\dive \cl{\uu'}$, whence $\cl{\uu'}\dive\Dt{\mm- 0.5}$ by Lemma~\ref{L:Divnn2}. We then find $\gh{\mm-1}\cl{\uu'}\dive \gh{\mm-1} \Dt{\mm-0.5} = \Dt\mm$. Therefore, $\cl\uu$ has the form $\gh\mm \gh{\mm-1} \gh\mm \cl{\uu'}$ with $\gh{\mm-1} \cl{\uu'}\dive\Dt\mm$. This means that $\cl\uu$ lies in~$\DDDD{\mm + 1}{\lgg\uu}$ and implies $\cl\uu \dive \Dt{\mm + 1}$, as expected. 

Thus, $\cl\uu \dive \Dt{\mm + 1}$ holds in every possible case, and this completes the proof.
\end{proof}

Merging Lemmas~\ref{L:Intp} and~\ref{L:ObsNorm}, we finally obtain:

\begin{prop}\label{P:FNSimp}
A word of~$\AHs$ is the normal form of a simple element of~$\HHp$ if, and only if, it contains no factor in~$\LI$ or~$\LIp$.
\end{prop}

Thus, the monoid~$\HHp$ gives rise to a Garside combinatorics that is quite similar to that of the Thompson monoid~$\FFp$. In both cases, we have a family of simple elements that is filtered by the sequence~$(\Dt\nn)_{\nn \ge 1}$, with finitely many elements below~$\Dt\nn$, namely~$2^{\nn - 1}$ in the case of~$\FFp$ and $2 \cdot 3^{\nn - 2}$ in the case of~$\HHp$, and the normal forms of simple elements are characterized in terms of finitely many types of forbidden factors of length~$2$ or~$3$, namely the factors~$\gf\ii \gf\jj$ with~$\jj \ge \ii$ in the case of~$\FFp$, and the factors~$\gh\ii \gh\jj$ with~$\jj \ge \ii + 2$ or~$\jj = \ii$ and the factors~$\gh\ii \gh{\ii + 1} \gh\jj$ with $\jj = \ii$ or $\jj = \ii + 2$ in the case of~$\HHp$.

However, the parallel is not complete, as, in the case of~$\HHp$, simple elements do not form a Garside family. Indeed, the element~$\gh2\gh4$ is not simple, although it right divides the simple element~$\gh1 \gh2 \gh4$, \ie,~$\Dt3$. It is easy to check that every element of~$\HHp$ admits a greatest simple left divisor, namely its greatest common left divisor with~$\Dt\nn$ for~$\nn$ sufficiently large, and, from there, to show for every element the existence of a greedy decomposition in terms of simple pieces, but the decompositions so obtained fail to obey the good properties that make Garside families interesting. In particular, the ``domino rule'' of~\cite[Prop.~III.1.45]{Dir}, implying that the elements of~$\HHp$ have no well defined degree in terms of simple elements.

The enveloping group of the monoid~$\HHp$ is the group~$\HH$ defined by the presentation~$\PH$. At this point, the most puzzling open problem about~$\HHp$ is

\begin{ques}\label{Q:Embed}
Does the monoid~$\HHp$ embed in the group~Ê$\HH$?
\end{ques}

The monoid~$\HHp$ is cancellative, but some pairs of elements of~$\HHp$ fail to admit a common left multiple, or a common right multiple, and, therefore, contrary to~$\FFp$ and~$\FF$, the group~$\HH$ is not a group of (left or right) fractions for~$\HHp$. As checking the Malcev conditions~\cite{ClP} for~$\HHp$ seems problematic, a more realistic way for proving that $\HHp$ embeds in~$\HH$ could be to construct a faithful representation of~$\HHp$ in a group of matrices. No such representation is known so far, but mapping~$\gh\ii$ to the surjection~$\FF_\ii$ from~$\ZZZZp$ to itself defined by 
$$\FF_\ii(\kk):= \kk \text{ for $\kk \le \ii + 1$},\ \FF_\ii(\ii + 2):= \ii, \text{\ and\ }\FF_\ii(\kk):= \kk - 1 \text{ for $\kk \ge \ii + 3$}$$
provides a representation~$\rho$ of~$\HHp$ that does not factor through~$\FFp$. The images of~$\gh1^2\gh2$ and $\gh1\gh2\gh3$ under~$\rho$ coincide, so $\rho$ is not faithful, but experiments reported in~\cite{Tes} suggest that the polynomial deformation~$\widetilde\rho$ of~$\rho$ that maps~$\gh\ii$ to the linear transfor\-mation~$\widetilde\FF_\ii$ defined by $\widetilde\FF_\ii(\vec\xx)_\kk:= \xx_\kk$ for $\kk \le \ii$, $\widetilde\FF_\ii(\vec\xx)_\kk:= \xx_{\kk - 1}$ for $\kk \ge \ii + 3$, plus
$$\widetilde\FF_\ii(\vec\xx)_{\ii + 1}:= \ttt\xx_\ii + (1 - \ttt)\xx_{\ii + 1}\quad \text{and}\quad \widetilde\FF_\ii(\vec\xx)_{\ii + 2}:= (1 + \ttt)\xx_\ii - \ttt\xx_{\ii + 1}$$
could be faithful. The involved matrices are not invertible, so proving that $\widetilde\rho$ is faithful would not solve Question~\ref{Q:Embed} directly, but it could be a promising first step.


\end{document}